\documentclass[11pt]{article}
\pdfoutput=1

\usepackage[T1]{fontenc}
\usepackage{lmodern}
\usepackage[a4paper, bottom = 3.5cm]{geometry}
\usepackage{mathtools}
\usepackage{amssymb}
\usepackage{amsthm}
\usepackage{color}
\usepackage[hidelinks]{hyperref}
\usepackage{enumerate}
\usepackage{microtype}
\usepackage{multirow}
\usepackage{booktabs} 
\usepackage{algorithm}
\usepackage{algorithmicx}
\usepackage{algpseudocode}
\algnewcommand\algorithmicinput{\textbf{Input:}}
\algnewcommand\Input{\item[\algorithmicinput]}
\algnewcommand\algorithmicoutput{\textbf{Output:}}
\algnewcommand\Output{\item[\algorithmicoutput]}

\newtheorem{theorem}{Theorem}[section]

\newtheorem{lemma}[theorem]{Lemma}
\newtheorem{corollary}[theorem]{Corollary}
\theoremstyle{definition}

\newtheorem*{definition*}{Definition}
\theoremstyle{remark}
\newtheorem{remark}[theorem]{Remark}
\theoremstyle{definition}

\newcommand{\R}{{\mathbb{R}}}
\newcommand{\Lip}{\lambda}
\newcommand{\frob}{\mathrm{F}}
\DeclareMathOperator{\tr}{trace}
\DeclareMathOperator{\rank}{rank}
\DeclareMathOperator{\grad}{grad}
\DeclareMathOperator{\qr}{qr}
\DeclareMathOperator*{\argmin}{argmin}
\newcommand{\abs}[1]{{\left\lvert #1 \right\rvert}}

\numberwithin{equation}{section}

\title{Gauss--Southwell type descent methods for\\low-rank matrix optimization}
\author{Guillaume Olikier\thanks{Université Côte d’Azur and Inria, Epione Project Team, 2004 route des Lucioles - BP 93, 06902 Sophia Antipolis Cedex, France} \and Andr\'e Uschmajew\thanks{Institute of Mathematics \& Centre for Advanced Analytics and Predictive Sciences, University of Augsburg, 86159 Augsburg, Germany} \and Bart Vandereycken\thanks{Section of Mathematics, University of Geneva, 1211 Geneva, Switzerland}}
\date{}

\begin{document}

\maketitle

\begin{abstract}
We consider gradient-related methods for low-rank matrix optimization with a smooth cost function. The methods operate on single factors of the low-rank factorization and share aspects of both alternating and Riemannian optimization. Two possible choices for the search directions based on Gauss--Southwell type selection rules are compared: one using the gradient of a factorized non-convex formulation, the other using the Riemannian gradient. While both methods provide gradient convergence guarantees that are similar to the unconstrained case, numerical experiments on a quadratic cost function indicate that the version based on the Riemannian gradient is significantly more robust with respect to small singular values and the condition number of the cost function. As a side result of our approach, we also obtain new convergence results for the alternating least squares method.
\end{abstract}

\section{Introduction}

We consider a differentiable function $f \colon \R^{m \times n} \to \R$ and the optimization problem
\begin{equation}\label{eq: optimization problem}
\min_{X \in \R^{m \times n}} f(X) \quad \text{s.t.} \quad \rank(X) \le k, 
\end{equation}
where $m$, $n$, and $k$ are positive integers such that $k < \min\{m,n\}$. The constraint set
\begin{equation}\label{eq: set M leq k}
\mathcal{M}_{\le k} \coloneqq \{ X \in \R^{m \times n} \colon \rank(X) \le k \}
\end{equation}
is non-convex, which makes it difficult to find the global minimum even if the function~$f$ would be convex. Nevertheless, very efficient first-order algorithms for finding critical points can be implemented based on alternating minimization.

At the core, these algorithms are derived from the fact that the rank constraint in~\eqref{eq: optimization problem} can be eliminated by using the bilinear factorization
\[
X = L R^\top, \quad L \in \R^{m \times k}, \ R \in \R^{n \times k}
\]
of a matrix of rank at most $k$. Problem~\eqref{eq: optimization problem} is hence equivalent to the unconstrained minimization problem
\begin{equation}\label{eq: opt problem factored form}
\min_{\substack{L \in \R^{m \times k} \\ R \in \R^{n \times k}}} g(L,R) \qquad \text{with} \qquad g(L,R) \coloneqq f(L R^\top).
\end{equation}
In the terminology of the recent work~\cite{LevinKileelBoumal2024}, the map $\R^{m \times k} \times \R^{n \times k} \to \R^{m \times n} : (L, R) \mapsto LR^\top$ is a smooth lift of the determinantal variety $\mathcal{M}_{\le k}$ \cite[Lecture~9]{Harris1992} and $g$ is the lifted cost function. The class of alternating optimization methods then consists of algorithms of the following general form:

\begin{enumerate}[(1)]\itemsep2pt
\item
Given $X_\ell = L_\ell^{} R_\ell^\top$ in factorized form, choose whether to update the $L$- or $R$-factor.
\item
Depending on the choice in (1) find an update $\hat L_{\ell+1}$ or $\hat R_{\ell+1}$ and set $X_{\ell+1} \coloneqq \hat L_{\ell+1}^{} R_\ell^\top$ or $X_{\ell+1} \coloneqq L_\ell^{} \hat R_{\ell+1}^\top$, respectively.
\item
Change the representation: choose $L_{\ell+1}$ and $R_{\ell+1}$ such that $X_{\ell+1} = L_{\ell+1}^{} R_{\ell+1}^\top$.
\end{enumerate}

In this paper we consider two different methods of this type (Algorithms~\ref{algo: balanced GD} and~\ref{algo: Subspace GD} below), where the search directions in step (2) are gradient-related in a suitable sense and step (1) is based on a Gauss--Southwell type selection rule.

The most obvious search directions are the negative partial gradients $- \nabla_L \, g(L_\ell, R_\ell)$ and $- \nabla_R \, g(L_\ell, R_\ell)$. Unfortunately, this straightforward choice suffers from the non-uniqueness of the representation $X = LR^\top$ in a twofold way. First, the resulting search direction in the image space $\R^{m \times n}$ depends on the particular chosen representation $X_\ell^{} = L_\ell^{} R_\ell^\top$ (see the formulas~\eqref{eq: partial gradients of g} below), which is conceptually unsatisfying. Second, the Lipschitz constant for the partial gradients depends on the chosen representation as well, which in theory may pose non-uniform step-size restrictions for ensuring convergence. And indeed, it can be observed in practice that this has a large influence on the convergence speed. Both problems, however, can be overcome by using a so-called balanced factorization in step (3) of the above pseudocode for removing the non-uniqueness (see section~\ref{sec: Balanced factorized version}), as has been proposed in several works; see, for example, \cite{Zhu2018,Zhu2021,Li2019,Li2020} where balancing is indirectly promoted via regularization. While such a regularization may have other potential benefits, we omit it here in order to not change the cost function and instead apply explicit balancing after every step, which comes at negligible~additional~cost.

Regarding the choice of the factor to be updated in step (1), the cyclic Gauss--Seidel type update rule is the most common choice and optimizes $L$ and $R$ alternatingly. For example, variants of the alternating gradient descent method (using the partial gradients of $g$) have been proposed and analyzed for low-rank matrix completion in~\cite{SunLuo2016} and~\cite{Tanner2016}. However, instead of a cyclic rule, in this work we advocate for using Gauss--Southwell type update rules in which the block variable with the largest partial gradient is selected for the next update. This selection rule originates from classic iterative solvers for linear systems~\cite{Saad2020} and has also been well known for nonlinear optimization ever since; see, e.g.,~\cite{LuoTseng1992,Nesterov2012,Nutini2015} for convex problems and~\cite{Nutini2022} for a general overview. The general advantage of the Gauss--Southwell rule is that the norm of the selected block gradient is at least a fraction of the norm of the full gradient of $g$, which allows to deduce convergence rate estimates for obtaining $\varepsilon$-critical points by simply applying the descent property of gradient descent (see~\eqref{eq: gradient descent decrease}) for the partial gradient steps. However, in the context of low-rank optimization, one known issue, also observed for other update rules, then still remains: the convergence rate in terms of function values or distances to solutions severely deteriorates for ill-conditioned solutions (accumulation points with a large ratio of the largest and $k$th largest singular value), one intrinsic reason being that the gradient of the function $g$ is not a robust measure of criticality on the manifold of rank-$k$ matrices.

As an alternative to the partial gradients of $g$, we will therefore also consider a different way of obtaining search directions, which is motivated from a Riemannian optimization perspective. The bilinear nature of the parametrization of the set $\mathcal{M}_{\le k}$ entails that at the smooth points~$X$ of that set (the matrices of rank exactly $k$), the \emph{tangent space} is decomposed into two overlapping linear spaces $T_1(X) + T_2(X)$, each of which contains $X$ and is itself entirely contained in $\mathcal{M}_{\le k}$. It is therefore natural to use the orthogonal projections of the negative gradient $- \nabla f(X_\ell)$ onto $T_1(X_\ell)$ or $T_2(X_\ell)$ as search directions. These search directions are \emph{independent} of the chosen factorization $X_\ell = L_\ell^{} R_\ell^\top$ and actually correspond to partial gradients of $g$ when the factor not to be updated is chosen to have orthonormal columns, which can be achieved by changing the representation using QR factorization after every update. Such a variant of alternating optimization including QR factorization is usually implemented anyway for reasons of numerical stability. As an alternative interpretation, taking orthogonal projections of $- \nabla f(X)$ onto $T_1(X)$ and $T_2(X)$ as search directions turns out to be mathematically equivalent to taking so-called~\emph{scaled} (or \emph{preconditioned}) variants $- \nabla_L \, g(L, R) \cdot (R^\top R)^{-1}$ and $- \nabla_R \, g(L, R) \cdot (L^\top L)^{-1}$ of the partial gradients of $g$ for updating $L$ and $R$; see~\eqref{eq:update scaledGD} below. Such scaled gradients have been considered for example in~\cite{Mishra2012,Tanner2016,TongMaChi2021}. Notably, in~\cite[\S 4]{Mishra2012} they are derived from choosing a suitable quotient geometry for the manifold of fixed-rank matrices. Corresponding optimization methods based on cyclic or simultaneous update rules are called~\emph{scaled alternating steepest descent} (ScaledASD) in~\cite{Tanner2016} and~\emph{scaled gradient descent} (ScaledGD) in~\cite{TongMaChi2021}. In these works, the accelerating effect and/or robustness with respect to small singular values is clearly confirmed; see also~\cite[Remark~7]{Luo2024}. While~\cite[Theorem~4.1]{Tanner2016} proves stationarity of accumulation points of ScaledASD for a non-convex formulation of matrix completion under minimal assumptions, but without rate,~\cite[Theorem~11]{TongMaChi2021} provides local linear convergence rates for ScaledGD close to global minima of general cost functions with rank-restricted convexity properties in noiseless scenarios, including matrix sensing under restricted isometry assumptions or matrix completion under incoherence assumptions.

The observation that the scaled partial gradients of $g(L,R)$ correspond to orthogonal projections of $\nabla f(X)$ onto linear subspaces within $\mathcal M_{\le k}$, although not difficult, could be new and gives these methods an interpretation as Riemannian optimization methods on the embedded rank-$k$ manifold. In particular, as we show in the present paper, when combined with a Gauss--Southwell rule, this viewpoint allows to analyze the convergence of the corresponding method (presented in Algorithm~\ref{algo: Subspace GD}) as a gradient-related Riemannian descent method. In fact, the search directions then satisfy an angle condition with the negative~\emph{Riemannian gradient} (here: the projection of $-\nabla f(X)$ onto the full tangent space) and hence the convergence rate to an $\varepsilon$-Riemannian critical point can be estimated in a quite simple way based on standard results. This is carried out in section~\ref{sec: embedded version}, the main results being a global $O(1/\sqrt{\ell})$ convergence rate for the Riemannian gradient. In terms of the above references, Algorithm~\ref{algo: Subspace GD} may be equivalently interpreted as a Gauss--Southwell variant of scaled gradient descent, and while our convergence result seems to be new in this general form, we do not claim (and did not test) that a Gauss--Southwell rule for scaled gradients leads to better performance than other variants. Instead, our focus in this work is on the simple convergence analysis that it admits, despite being a Riemannian method. We note that a method of this type has already been proposed as a retraction-free Riemannian method in~\cite[Algorithm~4]{Schneider2015}.

Having two possible choices of gradient-related search directions, it is then natural to ask which of the two is better. In terms of the rates, it turns out that the Riemannian version needs the same (estimated) number of steps for making the Riemannian gradient of $f$ smaller than $\varepsilon$ as the balanced factorized version needs for making the gradient of $g$ smaller than $\varepsilon$. However, as we show in section~\ref{sec: comparison}, the first condition implies the second (up to a small constant) but not the other way around. In this sense, we will confirm as one of the main messages of the paper that the Riemannian search directions are much more powerful than expected. Our numerical experiments in section~\ref{sec: numerical experiments} indeed confirm that it converges faster and exhibits stability with respect to small singular values of $X$, which is completely in line with aforementioned results on scaled gradient descent in~\cite{TongMaChi2021}.

One important thing to note about both of our methods is that limiting points can fail to be critical points of the original problem~\eqref{eq: optimization problem}, despite the Riemannian gradients converging to zero during the iteration. Although it is rarely observed in practice, such a behavior, termed an ``apocalypse'' in~\cite{LevinKileelBoumal2023}, can occur when a rank drop happens at the limit; cf.~Remark~\ref{rem: critical points} below. For the Riemannian version of our algorithm, it has been recently proposed how to exclude such undesirable points as accumulation points of the algorithms by disregarding the $k$th singular value if it is below a prescribed threshold~\cite[Algorithm~3]{OlikierAbsil2023}, while preserving the convergence rates for the Riemannian gradient as stated in our Corollary~\ref{cor: subspace}. In our work, we will largely ignore this subtlety (and assume full rank $k$ at accumulation points) in order to keep things conceptually simple.

There are two other sets of results in this work. In section~\ref{sec: Linear convergence rate}, we address the fact that the algebraic (i.e., sublinear) convergence rates for gradients obtained from a standard analysis based on Lipschitz continuity of the gradients are usually overly pessimistic. Indeed, in practice one observes a linear convergence rate. For the Riemannian version of our Gauss--Southwell approach we provide such a result in section~\ref{sec: Linear convergence rate} provided one of the accumulation points has a positive-definite Riemannian Hessian. This is a standard assumption for obtaining linear rates and for instance also implied by the assumptions of~\cite[Theorem~11]{TongMaChi2021} on local linear convergence of ScaledGD. Our proof is similar to the linear convergence proof for the Riemannian gradient method in~\cite[Theorem~4.5.6]{Absil2008}. In section~\ref{sec: ALS} we show how our arguments can be adapted for obtaining seemingly novel convergence estimates for the popular alternating least squares (ALS) algorithm. In fact, for two blocks as in low-rank matrix optimization, the cyclic rule of ALS coincides with a Gauss--Southwell  rule, since one of the partial gradients is always zero by construction. Our global $O(1/\sqrt{\ell})$ rate for the Riemannian gradients in ALS as presented in Corollary~\ref{cor: ALS} is insofar remarkable, as most of the available results for alternating minimization only deal with convex problems (which allows for better rates;~see, e.g.,~\cite{Hong2017} for an overview). Our arguments, however, share some similarity with the results in~\cite[Section~3.1]{Beck2015} on general non-convex functions, but are more specific to the low-rank optimization scenario by considering scaled and Riemannian gradients.

The remainder of this work is structured as follows. In section~\ref{sec: gradient descent methods} we review basic facts on the gradient method for $\Lip$-smooth functions on affine subspaces, which will serve as the main tool. Section~\ref{sec: main section} is the core section of this work. Section~\ref{sec: Balanced factorized version} develops the algorithm based on partial gradients of the function $g$ and balancing (Algorithm~\ref{algo: balanced GD}). The main results on gradient convergence for various step-size rules are Theorem~\ref{thm: balanced} and Corollary~\ref{cor: balanced}. In section~\ref{sec: embedded version}, the algorithm based on the decomposition of the Riemannian gradient (Algorithm~\ref{algo: Subspace GD}) is developed in an analogous fashion. Main results are Theorem~\ref{thm: subspace} and Corollary~\ref{cor: subspace}. In section~\ref{sec: comparison} both approaches are compared; corresponding numerical experiments are presented in section~\ref{sec: numerical experiments}. The results regarding local linear convergence and ALS are given in sections~\ref{sec: Linear convergence rate} and~\ref{sec: ALS}, respectively.

\section{Recap of gradient descent methods}\label{sec: gradient descent methods}

Let us first briefly recall some main results on the convergence of the gradient method for $\Lip$-smooth functions. While the results are very well known (see, e.g., \cite{Nesterov2018}), we present them here for the case of a differentiable function $\phi \colon T \to \R$ defined on an affine subspace $T$ of a Euclidean vector space. In our applications later, $\phi$ will be the restriction of the cost functions $f$ and $g$, respectively, to certain subspaces. Given $x \in T$, one step of the gradient method in the space $T$ takes the form
\begin{equation}\label{eq: gradient descent}
x_+ \coloneqq x - h \nabla \phi(x).
\end{equation}
Here $h > 0$ is a step size that can be chosen either fixed or adaptively at each iteration. The ultimate goal of the iteration is to find a point $x^* \in T$ satisfying $\nabla \phi(x^*) = 0$.

The easiest analysis of the gradient descent step~\eqref{eq: gradient descent} assumes that the function $\phi$ is \emph{$\Lip$-smooth}, which means that $\nabla \phi$ is Lipschitz continuous with constant $\Lip \ge 0$:
\[
\| \nabla \phi(x) - \nabla \phi(y) \| \le \Lip \| x - y \| \text{ for all } x, y \in T.
\]
The key estimate implied by this property is the following. If $\phi$ is $\Lip$-smooth on the line segment connecting $x$ and $y$, then
\begin{equation}\label{eq: key estimate}
\phi(x) - \phi(y) \ge \langle \nabla \phi(x), x - y \rangle - \frac{\Lip}{2} \| x - y \|^2.
\end{equation}
The proof is straightforward from a Taylor expansion with integral remainder; see, e.g.,~\cite[Lemma~1.2.3]{Nesterov2018}.

If applicable,~\eqref{eq: key estimate} yields the following estimate for one step of the gradient method:
\begin{equation}\label{eq: gradient descent decrease}
\phi(x) - \phi(x_+) \ge \left( h - \frac{\Lip}{2} h^2 \right) \| \nabla \phi(x) \|^2.
\end{equation}
The factor in front of $\| \nabla \phi(x) \|^2$ is positive for step sizes $h = \frac{2\alpha}{\Lip}$ with $\alpha \in (0,1)$, leading to
\begin{equation*}\label{eq: quasi optimal decrease}
\phi(x) - \phi(x_+) \ge \frac{2}{\Lip} \alpha (1 - \alpha) \| \nabla \phi(x) \|^2.
\end{equation*}
The ``optimal'' step size that maximizes the right-hand side in this estimate is
\[
h^* \coloneqq \frac{1}{\Lip}
\]
and yields
\begin{equation}\label{eq: optimal decrease}
\phi(x) - \phi(x_+) \ge \frac{1}{2\Lip} \| \nabla \phi(x) \|^2.
\end{equation}

The ``optimal'' step size $h^*$ has the disadvantage that it requires knowing a conservative estimate for $\Lip$, which may lead to $h^*$ being unnecessarily small. There are several adaptive strategies for addressing this. For instance, when $\phi$ is a convex quadratic function, it is possible to find a so-called ``exact'' step size that minimizes the function $h \mapsto \phi(x - h \nabla \phi(x))$. Using this step size we obviously obtain the same estimate as~\eqref{eq: optimal decrease}, given that the $\Lip$-smoothness condition is satisfied.

A more general strategy is line search. Here the Armijo step-size rule is among the most well known: given $\gamma \in (0,1)$ find a step size $h > 0$ satisfying the inequality
\begin{equation}\label{eq: Armijo condition}
\phi(x) - \phi(x_+) \ge h \gamma \| \nabla \phi(x) \|^2.
\end{equation}
To avoid an unnecessarily small $h$, one typically uses a backtracking procedure by checking step sizes of the form
\[
h = \bar{h}, \bar{h} \beta, \bar{h} \beta^2, \dots
\]
in decreasing order with some fixed $\bar{h} \in (0,\infty)$ and $\beta \in (0,1)$. By~\eqref{eq: gradient descent decrease}, the Armijo condition~\eqref{eq: Armijo condition} will be satisfied once
\(
h - \frac{\Lip}{2} h^2 \ge h \gamma,
\)
or equivalently,
\begin{equation}\label{eq: stop backtrack}
h \le \frac{2}{\Lip}(1 - \gamma).
\end{equation}
Therefore, the backtracking procedure will terminate with $h = \bar{h} \beta^\ell$ for some
\[
\ell \le \max\left\{0, \left\lceil \ln \left( \frac{2}{\bar{h}\Lip}(1 - \gamma) \right) / \ln \beta \right\rceil \right\}.
\]
To avoid the rounding above, we note that $h \ge \min\{\frac{2}{\Lip} \beta (1 - \gamma), \bar{h}\}$ since we stop backtracking as soon as~\eqref{eq: stop backtrack} is satisfied. The resulting estimate~\eqref{eq: Armijo condition} is
\[
\phi(x) - \phi(x_+) \ge \min\left\{\frac{2}{\Lip} \beta \gamma (1-\gamma), \bar{h}\gamma \right\} \| \nabla \phi(x) \|^2.
\]

For later reference we emphasize that in this point-wise estimate for the Armijo step size it is in principle possible to replace a global smoothness constant $\Lip$ with a restricted smoothness constant on the line segment connecting $x$ and $x - \bar{h} \nabla f(x)$, that is, with a constant $\Lip_x \le \Lip$ satisfying
\begin{equation}\label{eq: restricted smoothness}
\| \nabla \phi(x) - \nabla \phi(x - t \bar{h} \nabla \phi(x)) \| \le \Lip_x t \bar{h}\| \nabla \phi(x) \|
\end{equation}
for all $0 \le t \le 1$. In this sense the Armijo step-size rule with backtracking automatically detects the locally optimal smoothness constant. We will make use of this fact in the proof of Theorem~\ref{thm: linear rate 2} further below.

\section{Block gradient descent for low-rank matrices}\label{sec: main section}

In this main section we show how to apply the above framework of gradient descent on subspaces to low-rank matrix optimization problems. We first consider the problem in the factorized form~\eqref{eq: opt problem factored form}. Here the goal is to minimize the unconstrained function
\[
g(L,R) \coloneqq f(LR^\top)
\]
for $L \in \R^{m \times k}$ and $R \in \R^{n \times k}$. The most obvious first-order method based on alternating optimization consists in taking gradient steps along the $L$ or $R$ block only, that is, taking $- \nabla_L \, g(L, R)$ or $- \nabla_R \, g(L, R)$ as a search direction. Depending on which block variable is chosen, we then update the point $X = L R^\top$ by using either
\begin{equation}\label{eq: update L with gradient}
 L_+ \coloneqq L - h \nabla_L \, g(L, R), \qquad R_+ \coloneqq R, 
\end{equation}
or
\begin{equation}\label{eq: update R with gradient}
L_+ \coloneqq L, \qquad R_+ \coloneqq R - h \nabla_R \, g(L, R).
\end{equation}
Observe that the first update~\eqref{eq: update L with gradient} is in fact a gradient step for the function $g$ when restricted to
\begin{equation}\label{eq: space T_R}
{\hat T}_R \coloneqq \{ (\hat L, R) \colon \hat L \in \R^{m \times k} \}
\end{equation}
with $R$ fixed. Likewise, the second update~\eqref{eq: update R with gradient} is a gradient step for $g$ restricted to
\begin{equation}\label{eq: space T_L}
\hat T_L \coloneqq \{ (L, \hat R) \colon \hat R \in \R^{n \times k}\}
\end{equation}
with $L$ fixed. Since $\hat T_R$ and $\hat T_L$ are \emph{affine} subspaces of $\R^{m \times k} \times \R^{n \times k}$, we can analyze these updates using the results from section~\ref{sec: gradient descent methods} by taking $\phi \coloneqq g|_{\hat T_R}$ and $\phi \coloneqq g|_{\hat T_L}$, respectively. 

To this end, let us estimate the Lipschitz constants of $\nabla_L \, g(L, R)$ and $\nabla_R \, g(L, R)$ on these affine subspaces. Note that with $X = LR^\top$ the partial gradients are given as
\begin{equation}\label{eq: partial gradients of g}
\nabla_L \, g(L, R) = \nabla f(X) R, \qquad \nabla_R \, g(L, R) = \nabla f(X)^\top L.
\end{equation}
We assume that $f$ is $\Lip$-smooth for the Frobenius norm $\| \cdot \|_\frob$ on $\R^{m \times n}$. Then
\begin{equation}\label{eq: Lipschitz constant 1}
\| \nabla_L \, g(L_1, R) - \nabla_L \, g(L_2, R) \|_\frob \le \Lip \| L_1 R^\top - L_2 R^\top \|_\frob \| R \|_2 \le \Lip \| R \|_2^2 \| L_1 - L_2 \|_\frob
\end{equation}
for all $L_1,L_2 \in \R^{m \times k}$, that is, the function $L \mapsto \nabla_L \, g(L, R)$ has a Lipschitz constant $\Lip \| R \|_2^2 $ on $\hat T_R$. Here $\|R\|_2$ denotes the spectral norm. Similarly, we have the estimate
\begin{equation}\label{eq: Lipschitz constant 2}
\| \nabla_R \, g(L, R_1) - \nabla_R \, g(L, R_2) \|_\frob \le \Lip \| L \|_2^2 \| R_1 - R_2 \|_\frob
\end{equation}
for all $R_1,R_2 \in \R^{n \times k}$, that is, we obtain the Lipschitz constant $\Lip \| L \|_2^2$.

Observe, however, the factors $L$ and $R$ in the representation of $X = LR^\top$ are not unique, and so the Lipschitz constants depend on the chosen representation. In order to get independent constants some additional normalization will be necessary. In the following we discuss two approaches. The first one, discussed in the next section, is based on balancing the singular values among the factors. The second approach, discussed afterwards, uses orthogonalization and leads to a Riemannian interpretation.

\subsection{Balanced factorized version}\label{sec: Balanced factorized version}

The low-rank representation $X = L R^\top$ is said to be \emph{balanced} if $L$ and $R$ have the same singular values. Given an economy-sized SVD $X = U \Sigma V^\top$, this is accomplished by choosing 
\[
L \coloneqq U \Sigma^{1/2}, \qquad R \coloneqq V \Sigma^{1/2}.
\]
For a balanced representation, the estimates~\eqref{eq: Lipschitz constant 1} and~\eqref{eq: Lipschitz constant 2} become
\begin{equation}\label{eq: Lipschitz for L}
\| \nabla_L \, g(L_1, R) - \nabla_L \, g(L_2, R) \|_\frob \le \Lip \| X \|_2 \| L_1 - L_2 \|_\frob
\end{equation}
and
\begin{equation}\label{eq: Lipschitz for R}
\| \nabla_R \, g(L, R_1) - \nabla_R \, g(L, R_2) \|_\frob \le \Lip \| X \|_2 \| R_1 - R_2 \|_\frob,
\end{equation}
that is, we get the Lipschitz constant $\lambda \| X \|_2$ for both partial gradients.

As mentioned above, our goal is to apply the results from section~\ref{sec: gradient descent methods} to bound the decrease in function value $f(X) - f(X_+)$ after one gradient step in either $L$ or $R$ direction. This will result in a lower bound in terms of $\| \nabla_L \, g(L, R) \|_\frob$ or $\| \nabla_R \, g(L, R) \|_\frob$ using suitable step sizes. If we now choose the block variables $L$ and $R$ according to which partial gradient is larger, we can also bound this function value decrease by the (Euclidean) norm of the full gradient, which will be denoted as
\[
\| \nabla g(L,R) \|_\frob \coloneqq \sqrt{\|\nabla_L \, g(L, R)\|_\frob^2 + \|\nabla_R \, g(L, R)\|_\frob^2}.
\]
Such a selection rule is called a Gauss--Southwell rule in the context of block coordinate descent (BCD) methods. 

To repeat the process, the factors need to be balanced after every update. This requires to compute an SVD of a $k \times k$ matrix in every step, which is cheap if $k$ is small. The final algorithm is displayed in Algorithm~\ref{algo: balanced GD}.

\begin{algorithm}[H]
\small
\caption{Balanced block gradient descent with Gauss--Southwell rule}
\label{algo: balanced GD}
\begin{algorithmic}[1]
\Input
$L_0^{} \in \R^{m \times k}$, $R_0 \in \R^{n \times k}$ balanced
\For
{$\ell = 0,1,2,\dots$}
\State
Compute $G \coloneqq \nabla f(L_\ell^{} R_\ell^\top)$ and partial gradients:
\[
(G_1, G_2) \coloneqq \nabla g (L_\ell, R_\ell) =  (G R_\ell, G^\top L_\ell).
\]
\State
Perform a gradient step according to Gauss--Southwell rule with suitable step $h_\ell$:
\[
(L_+,R_+) \coloneqq \begin{cases} (L_\ell - h_\ell G_1, R_\ell) &\quad \text{if $\|G_1 \|_\frob \geq \| G_2 \|_\frob$,} \\ (L_\ell, R_\ell - h_\ell G_2) &\quad \text{else.} \end{cases}
\]
\State
Compute the compact QR factorizations $Q_1 S_1 = L_+$ and $Q_2 S_2 = R_+$, the SVD factorization $U \Sigma V^\top = S_1^{} S_2^\top$, and balance:
\[
L_{\ell+1} \coloneqq Q_1 U \Sigma^{1/2}, \qquad  R_{\ell+1} \coloneqq Q_2 V \Sigma^{1/2}.
\]
\EndFor
\end{algorithmic}
\end{algorithm}

We can make the following convergence statements for this method, depending on the type of step-size rule that is applied.

\begin{theorem}\label{thm: balanced}
Assume that $f \colon \R^{m \times n} \to \R$ is $\Lip$-smooth and that the constrained sublevel set $N_0 \coloneqq \{X \in \mathcal{M}_{\le k} : f(X) \le f(L_0^{} R_0^\top) \}$ is contained in a ball $\{X \in \R^{m \times n} \colon \| X \|_2 \le \rho \}$. Depending on the chosen step size in the $\ell$th step, the iterates of Algorithm~\ref{algo: balanced GD} satisfy 
\begin{equation}\label{eq: general estimate balanced}
g(L_\ell,R_\ell) - g(L_{\ell+1},R_{\ell+1}) \ge \frac{\vartheta_\ell}{2 \Lip \rho} \| \nabla g(L_\ell,R_\ell) \|_\frob^2
\end{equation}
where:
\begin{itemize}
\item
$\vartheta_\ell \coloneqq 2 \alpha_\ell (1 - \alpha_\ell)$ in case a fixed step size $h_\ell = \frac{2 \alpha_\ell}{\Lip \rho}$ with $\alpha_\ell \in (0,1)$ is used;
\item
$\vartheta_\ell \coloneqq \frac{1}{2}$ in case the step size is obtained by exact line search (if well defined);
\item
$\vartheta_\ell \coloneqq \min\left\{2 \beta \gamma (1-\gamma), \Lip \rho \gamma \bar{h} \right\}$ in case the step size is obtained from backtracking $h_\ell = \bar{h}, \bar{h}\beta, \bar{h}\beta^2, \dots$ until the Armijo condition
\[
g(L_\ell,R_\ell) - g(L_{\ell+1},R_{\ell+1}) \ge \gamma h_\ell \| G_{1/2} \|_\frob^2
\]
is fulfilled (here $\beta,\gamma \in (0,1)$ and $G_{1/2}$ denotes $G_1$ or $G_2$ depending on the block selected in this step). The backtracking procedure terminates with $h_\ell = \bar{h} \beta^i$ for some $i \le \max\left\{0, \left\lceil \ln \left( \frac{2}{\bar{h} \Lip \rho}(1 - \gamma) \right) /  \ln \beta \right\rceil \right\}$.
\end{itemize}
\end{theorem}

\begin{corollary}\label{cor: balanced}
Under the conditions of Theorem~\ref{thm: balanced}, let $(L_\ell,R_\ell)$ be the iterates of Algorithm~\ref{algo: balanced GD} using any combination of the considered step-size rules. However, if step sizes $h_\ell = \frac{2 \alpha_\ell}{\Lip\rho}$ from the first step-size rule are used infinitely often, assume (for this subsequence) $\inf_\ell \alpha_\ell > 0$ and $\sup_\ell \alpha_\ell < 1$. Then the generated sequence $g(L_\ell,R_\ell) = f(L_\ell^{} R_\ell^\top)$ of function values is monotonically decreasing and converges to some value $f_* \ge \min_{X \in \mathcal{M}_{\le k}} f(X)$. Moreover, $\nabla g(L_\ell,R_\ell) \to 0$ and the sequence $(L_\ell,R_\ell)$ has at least one accumulation point. Every accumulation point $(L_*,R_*)$ satisfies $g(L_*,R_*) = f_*$ and $\nabla g(L_*,R_*) = 0$. For every nonnegative integer $j$ it holds that
\begin{equation}\label{eq: min gradient}
\min_{0 \le \ell \le j} \| \nabla g(L_\ell,R_\ell) \|_\frob \le  \left( \frac{2\Lip \rho}{\vartheta} \cdot \frac{f(X_0) - f_*}{j+1} \right)^{1/2},
\end{equation}
where $\vartheta \coloneqq \inf_\ell \vartheta_\ell > 0$. 
In particular, given $\varepsilon > 0$ the algorithm returns an $\varepsilon$-critical point, i.e., a point $(L, R)$ such that $\| \nabla g(L, R) \|_\frob \le \varepsilon$, after at most $ \lceil 2\Lip \rho (f(X_0) - f_*)\vartheta^{-1} \varepsilon^{-2} - 1 \rceil$ iterations.
\end{corollary}

Note that $0 < \vartheta_\ell \le \frac{1}{2}$ in all cases of Theorem~\ref{thm: balanced} and hence $0 < \vartheta \le \frac{1}{2}$ in Corollary~\ref{cor: balanced}.

\begin{remark}\label{rem: critical points}
Before giving the proofs, we make some remarks. The first is that $\lambda$-smoothness is in principle only required on the sublevel set $N_0$ of the starting point. In addition, the exact line search is well defined since $N_0$ is bounded.

The second remark is that in the version with exact line search one could, as an a~posteriori estimate, replace for given $\varepsilon > 0$ and large enough $\ell = \ell(\varepsilon)$ the constant $\rho$ in~\eqref{eq: general estimate balanced} with $\max \| X_* \|_2 + \varepsilon$ where $\max \| X_* \|_2$ is the maximum spectral norm among all accumulation points. (For the other two versions the accumulation points depend on $\rho$ through step sizes.)

Finally, as mentioned in the introduction, the result $\nabla g(L_*,R_*) = 0$ does not necessarily imply that $X_*^{} = L_*^{}R^\top_*$ is a critical point of $f$ on the variety $\mathcal M_{\le k}$ (in the sense that the tangent cone to $\mathcal M_{\le k}$ at $X_*$ contains no descent direction for $f$). By \cite[Proposition~2.8]{LevinKileelBoumal2024}, we know that the implication holds if $\rank(L_*) = \rank(R_*) = k$ but may not hold otherwise. If the implication does not hold, $X_*$ is only critical for $f$ on the manifold of matrices with fixed rank $s \coloneqq \rank(X_*) < k$. Such a triplet $(X_*, (L_\ell^{} R_\ell^\top)_{\ell \in \mathbb{N}}, f)$ is called an apocalypse in~\cite{LevinKileelBoumal2023}. An apocalypse-free version of Algorithm~\ref{algo: Subspace GD} based on the original method from~\cite[Algorithm~4]{Schneider2015} has been recently proposed in~\cite[Algorithm~3]{OlikierAbsil2023}.
\end{remark}

\begin{proof}[Proof of Theorem~\ref{thm: balanced}]
Let $(L_\ell, R_\ell)$ be balanced such that $X_\ell^{} = L_\ell^{} R_\ell^\top \in N_0$. As noted above, the single updates are gradient descent steps for the function $g$ restricted to the subspaces $\hat T_{L_\ell}$ and $\hat T_{R_\ell}$, respectively, defined in~\eqref{eq: space T_R} and~\eqref{eq: space T_L}. By~\eqref{eq: Lipschitz for L} and~\eqref{eq: Lipschitz for R}, these restrictions are $(\Lip\rho)$-smooth and the results of section~\ref{sec: gradient descent methods} allow us to conclude that
\[
f(X_\ell) - f(X_{\ell+1}) = g(L_\ell,R_\ell) - g(L_{\ell+1},R_{\ell+1}) \ge \frac{\vartheta_\ell}{\Lip \rho} \| G_{1/2} \|_\frob^2
\]
with $\vartheta_\ell$ determined by the step-size rule as specified in the statement of the theorem. Since $\| \nabla g(L_\ell,R_\ell) \|_\frob^2 = \| G_1 \|_\frob^2 + \| G_2 \|_\frob^2$, the maximum block selected by the Gauss--Southwell rule satisfies $\| G_{1/2} \|_\frob^2 \ge \frac12\| \nabla g(L_\ell,R_\ell) \|_\frob^2$. This shows~\eqref{eq: general estimate balanced} after induction since $X_{\ell+1}^{} = L_{\ell+1}^{} R_{\ell+1}^\top \in N_0$ by the above estimates.
\end{proof}

\begin{proof}[Proof of Corollary~\ref{cor: balanced}]
Since the sequence $g(L_\ell, R_\ell) = f(X_\ell)$ is monotonically decreasing and bounded from below (since $N_0$ is bounded), it converges to $f_* \coloneqq \inf_{\ell \in \mathbb{N}} f(X_\ell)$. Thus, for every $j \in \mathbb{N}$, it holds that
\[
f(X_0) - f_* \ge f(X_0) - f(X_{j+1}) = \sum_{\ell=0}^{j} f(X_\ell) - f(X_{\ell+1}) \ge \sum_{\ell=0}^{j} \frac{\vartheta}{2\Lip\rho} \| \nabla g(L_\ell,R_\ell) \|_\frob^2,
\]
using the estimates from Theorem~\ref{thm: balanced}. Since $j$ is arbitrary, this shows that $\nabla g(L_\ell,R_\ell) \to 0$. On the other hand, there must be at least one $0 \le \ell \le j$ satisfying
\[
\frac{\vartheta}{2\Lip \rho}\| \nabla g(L_\ell,R_\ell) \|_\frob^2 \le \frac{f(X_0) - f_*}{j+1}.
\]
This shows~\eqref{eq: min gradient}. Finally, since $N_0$ is bounded and $\| L_\ell \|_2 = \| R_\ell \|_2 = \| X_\ell \|_2^{1/2}$, the sequence $(L_\ell,R_\ell)$ is bounded and thus has at least one accumulation point~$(L_*,R_*)$. By continuity, $g(L_*,R_*) = f_*$ and $\nabla g(L_*,R_*) = 0$.
\end{proof}

\subsection{Embedded Riemannian version}\label{sec: embedded version}

As derived in~\eqref{eq: partial gradients of g}, the search directions in a single step of the ``vanilla'' block gradient descent method for the function $g(L,R) \coloneqq f(LR^\top)$ depend on the chosen factorization $X = LR^\top$. In Algorithm~\ref{algo: balanced GD} this ambiguity is resolved using a balancing, which is a reasonable approach, but arguably still somewhat arbitrary. In particular, it leads to the appearance of the spectral norm of $X$ in the Lipschitz constants (see~\eqref{eq: Lipschitz for L} and~\eqref{eq: Lipschitz for R}) and thus in the convergence estimates (through the constant $\rho$). If instead of balancing one could choose the factor that is currently \emph{not} optimized to have orthonormal columns, this artefact would disappear (see~\eqref{eq: Lipschitz constant 1} and~\eqref{eq: Lipschitz constant 2}). In a cyclic update scheme for the factors this can be easily achieved and has been considered in many works. However, for the Gauss--Southwell rule it is not so clear how to combine it with orthogonalization. And even so, the partial gradients of $g$ would still depend on the choice of the orthogonal factor, which is not unique.

A natural solution to these issues with the factorized approach can be obtained by taking a subspace viewpoint in the actual ambient space $\R^{m \times n}$. For this, besides the low-rank factorization $X = LR^\top$ we also consider in the sequel an SVD-like decomposition $X = USV^\top$ where $S \in \R^{k \times k}$ and $U$ and $V$ have pairwise orthonormal columns. Assume $\rank(X) = k$. We now consider the subspace of matrices whose row space is contained in the row space of $X$,
\begin{equation*}\label{eq: subspace TV}
T_V \coloneqq \{ \hat L V^\top \colon \hat L \in \R^{m \times k} \} \subseteq \R^{m \times n},
\end{equation*}
and likewise for the column space,
\begin{equation*}\label{eq: subspace TU}
T_U \coloneqq \{ U \hat R^\top \colon \hat R \in \R^{n \times k} \} \subseteq \R^{m \times n}.
\end{equation*}
Observe that $T_V$ and $T_U$ are \emph{linear} subspaces contained in $\mathcal M_{\le k}$ and both of them contain $X$. They are the respective images of the affine subspaces $\hat{T}_R$ and $\hat{T}_L$ from~\eqref{eq: space T_R} and~\eqref{eq: space T_L} under the map $\R^{m \times k} \times \R^{n \times k} \to \R^{m \times n} : (\hat{L}, \hat{R}) \mapsto \hat{L}\hat{R}^\top$. (When $\rank(X) < k$ some subtleties regarding the choice of $U$ and $V$ occur.) However, contrary to $\hat{T}_R$ and $\hat{T}_L$, the subspaces $T_V$ and $T_U$ are uniquely determined by $X$ (in case $\rank(X) = k$) and do not depend on any fixed factorization: even if $V$ is not uniquely defined from $X=USV^\top$, any equivalent choice is of the form $VQ$ for some orthogonal $Q$ since $\rank(X) = k$, and obviously it then holds $T_{VQ} = T_V$. The same observation holds for $U$ and $T_U$.

Any optimization method based on sequential optimization of factors $L$ and $R$ in $X = L R^\top$ implicitly selects a search direction in the linear space $T_V$ or $T_U$. For instance, by~\eqref{eq: partial gradients of g}, the negative partial gradient $-\nabla_L \, g(L, R) = -\nabla f(X)R$ can be identified with the negative gradient of the restriction of $g$ to $\hat{T}_R$, namely $(-\nabla f(X)R, R)$, which is an element of $\hat{T}_R$ whose corresponding element in $T_V$ is $-\nabla f(X)RR^\top$. In line with what was noted above, this particular direction depends on the choice of $R$ (and $L$) in the decomposition $X = LR^\top$.

Gradient-related search directions that are independent from the factorization can be obtained based on information of the gradient $\nabla f(X)$ in the ambient space only. A natural choice is to consider the best approximations of $- \nabla f(X)$ in $T_V$ or $T_U$ w.r.t.~the Frobenius inner product. These can be easily computed as the orthogonal projections of $- \nabla f(X)$ at $X = USV^\top$ onto $T_V$ or $T_U$ using the formulas
\begin{equation*}\label{eq: projection T1 and T2}
\mathcal P_{V}(Z) = Z V V^\top, \qquad
\mathcal P_{U}(Z) = U U^\top Z.
\end{equation*}
This yields the following type of iteration in the set $\mathcal{M}_{\le k}$:
\begin{itemize}\itemsep0pt
	\item[(1)] Given $X = USV^\top \in \mathcal{M}_{\le k}$, choose an update direction: 
	\[
	P_1 \coloneqq \mathcal P_{V} (\nabla f(X)) = \nabla f(X) V V^\top \quad \text{or} \quad P_2 \coloneqq \mathcal P_{U} (\nabla f(X)) = U U^\top \nabla f(X).
	\]	 
	\item[(2)] Perform the step:
	$X_+ \coloneqq X - h P_{1/2}$.
\end{itemize}

In terms of a factorization $X = LR^\top$, the two possible updates $X_+$ in this iteration can be expressed using pseudo-inverses and~\eqref{eq: partial gradients of g} as (assuming $\rank(X) = k$)
\begin{align*}
X_+ &\coloneqq LR^\top - h \nabla f(X) R R^\dagger  = [L - h \nabla_L \, g(L, R) \cdot (R^\top R)^{-1}] R^\top \\ \text{or} \qquad  X_+ &\coloneqq LR^\top - h L L^\dagger \nabla f(X) = L [R - h \nabla_R \, g(L, R) \cdot (L^\top L)^{-1} ]^\top.
\end{align*}
As announced in the introduction, this shows that the Gauss--Southwell updates implicitly perform the updates
\begin{equation}\label{eq:update scaledGD}
L_+ \coloneqq L - h \nabla_L \, g(L, R) \cdot (R^\top R)^{-1} \quad \text{or} \quad R_+ \coloneqq R - h \nabla_R \, g(L, R) \cdot (L^\top L)^{-1}
\end{equation}
for the factors $L$ and/or $R$, which coincide with the well-known~\emph{scaled gradient descent} update formulas; see, e.g.,~\cite[(3)]{TongMaChi2021}.

Obviously, the projected gradients $P_1$ and $P_2$ equal the gradients of the restriction of~$f$ to the linear subspaces $T_V$ and $T_U$, respectively. Furthermore, if $f$ is $\Lip$-smooth in Frobenius norm, then so is its restriction to any linear subspace. This means we can apply the abstract results recorded in section~\ref{sec: gradient descent methods} to analyze one step of such a method and obtain an estimate
\begin{equation}\label{eq: estimate subspace}
f(X) - f(X_+) \ge \frac{\vartheta}{\lambda} \| P_{1/2} \|_\frob^2,
\end{equation}
where the value of $\vartheta$ depends on the step-size rule and is specified in Theorem~\ref{thm: subspace} below.

The estimate~\eqref{eq: estimate subspace} indicates that a Gauss--Southwell rule could be useful in deciding whether $P_1$ or $P_2$ should be selected for an update direction. This has been proposed in~\cite[Algorithm~4]{Schneider2015} in a slightly more elaborate setup for dealing with rank-deficient iterates. Indeed, we can motivate the Gauss--Southwell rule from the perspective of Riemannian optimization as follows. It is well known that if $X = USV^\top$ has rank $k$, then the linear space
\begin{equation*}\label{eq: tangent space}
T_{U,V} \coloneqq T_U + T_V
\end{equation*}
can be identified with the tangent space to the embedded submanifold $\mathcal{M}_k$ of rank-$k$ matrices at~$X$. From a geometric perspective, our goal in low-rank optimization methods is to drive the Riemannian gradient of $f$ to zero. When taking the embedded metric the following definition applies.

\begin{definition*}
The \emph{Riemannian gradient} of $f$ at $X \in \R^{m \times n}$, denoted by $\grad f(X)$, is the orthogonal projection of $\nabla f(X)$ onto the tangent space to $\mathcal M_{\rank(X)}$ at $X$.
\end{definition*}

Note that in this definition the fixed-rank manifold $\mathcal M_{\rank(X)}$ to which the Riemannian gradient relates is implicitly determined by $X$. Hence, when $\rank(X) = k$, it holds
\begin{equation*}\label{eq: Riemannian gradient}
\grad f(X) = \mathcal P_{U,V} (\nabla f(X) ),
\end{equation*}
where $\mathcal P_{U,V}$ is the orthogonal projector onto $T_{U,V}$, given by
\begin{equation*}\label{eq: projection on tangent space}
\mathcal P_{U,V}(Z) = Z VV^\top + UU^\top Z - U U^\top Z V V^\top,
\end{equation*}
with
\[
 \| \mathcal P_{U,V}(Z) \|_\frob^2 = \| Z VV^\top \|_\frob^2 + \| UU^\top Z \|_\frob^2 - \|U U^\top Z V V^\top \|_\frob^2.
\]
Applied to $\nabla f(X) $, this then implies
\[
\| \mathcal P_{U,V} (\nabla f(X) ) \|_\frob^2 \le \| \nabla f(X) V V^\top \|_\frob^2 + \| U U^\top \nabla f(X) \|_\frob^2 = \| P_1 \|_\frob^2 + \| P_2 \|_\frob^2,
\]
which leads to the important observation 
\[
\max\{\| P_1 \|_\frob^2, \| P_2 \|_\frob^2\} \ge \frac{1}{2} \|  \mathcal P_{U,V} (\nabla f(X) ) \|_\frob^2.
\]
As a result, employing a Gauss--Southwell rule will allow to turn the estimate~\eqref{eq: estimate subspace} into
\begin{equation}\label{eq: decrease Riemannian}
f(X) - f(X_+) \ge \frac{\vartheta}{2 \lambda} \| \grad f(X) \|_\frob^2.
\end{equation}
Based on this, a method with Gauss--Southwell rule allows for convergence statements regarding the Riemannian gradient without requiring any curvature bounds. The result is stated in Theorem~\ref{thm: subspace} below.

Some subtleties arise when $X = USV^\top$ has rank less than $k$ but the matrix $S$ has size $k \times k$. This can happen in our algorithm since the sizes of the decomposition are fixed. It is not difficult to show that the corresponding tangent space to $\mathcal M_{\rank(X)}$ at $X$ is then contained in the space $T_{U,V}$. Hence we still get the relation
\[
\| \mathcal P_{U,V} (\nabla f(X) ) \|_\frob \ge \|\grad f (X) \|_\frob
\]
so that~\eqref{eq: decrease Riemannian} is valid for all $X = USV^\top$ when using the Gauss--Southwell rule. 

While we have derived the method in the ambient space $\R^{m \times n}$, its practical implementation is of course done in the factorized form to keep benefit of the low-rank representation and avoid forming potentially huge matrices. Note in this context that we can exploit $\| \nabla f(X) V V^\top \|_\frob = \| \nabla f(X) V \|_\frob$ and $\| U U^\top \nabla f(X)  \|_\frob = \| U^\top \nabla f(X) \|_\frob$ for deciding the Gauss--Southwell rule. In this form the algorithm is outlined in Algorithm~\ref{algo: Subspace GD}, where $\qr(Y) = (Q,R)$ denotes a QR decomposition of $Y = QR$ with $k$ columns. For every iterate of rank $k$, a step of Algorithm~\ref{algo: Subspace GD} is equivalent to a step of \cite[Algorithm~4]{Schneider2015}.

\begin{algorithm}[H]
\small
\caption{Riemannian block gradient descent with Gauss--Southwell rule}
\label{algo: Subspace GD}
\begin{algorithmic}[1]
\Input
SVD-like decomposition $X_0 = U_0^{} S_0^{} V_0^\top$ where $U_0 \in \R^{m \times k}$ and $V_0 \in \R^{n \times k}$, $U_0^\top U_0^{} = V_0^\top V_0^{} = I_k$.
\For
{$\ell = 0,1,2,\dots$}
\State
Compute gradient $G \coloneqq \nabla f(U_\ell^{} S_\ell^{} V_\ell^\top)$ and
\[
(\hat P_1, \hat P_2) \coloneqq (G V_\ell, G^\top U_\ell).
\]
\State
Perform a gradient step according to the Gauss--Southwell rule with suitable step size $h_\ell$:
\If
{$\| \hat P_1 \|_\frob \ge \| \hat P_2 \|_\frob$}
\State
\[
(U_{\ell+1}, S_{\ell+1}) \coloneqq \qr(U_\ell S_\ell - h_\ell \hat P_1),\qquad V_{\ell+1} \coloneqq V_\ell,
\]
\Else
\State
\[
U_{\ell+1} \coloneqq U_\ell,\qquad (V_{\ell+1},S_{\ell+1}) \coloneqq \qr(V_\ell S_\ell - h_\ell \hat P_2).
\]
\EndIf
\EndFor
\end{algorithmic}
\end{algorithm}

The convergence statements for this method parallel those in Theorem~\ref{thm: balanced}. The difference is that they relate to the projected (Riemannian) gradients $\mathcal P_{U,V}(\nabla f(X))$ instead of $\nabla g(L,R)$, and only depend on the Lipschitz constant of $\nabla f(X)$. 

\begin{theorem}\label{thm: subspace}
Assume $f \colon \R^{m \times n} \to \R$ is $\Lip$-smooth and bounded from below. Depending on the chosen step size in the $\ell$th step, the iterates $X_\ell^{} = U_\ell^{} S_\ell^{} V_\ell^\top$ of Algorithm~\ref{algo: Subspace GD} satisfy 
\begin{equation}\label{eq: general estimate subspace}
f(X_\ell) - f(X_{\ell+1}) \ge \frac{\vartheta_\ell}{2\Lip} \| \mathcal{P}_{U_\ell,V_\ell} ( \nabla f(X_\ell)) \|_\frob^2 \ge \frac{\vartheta_\ell}{2\Lip} \| \grad f(X_\ell) \|_\frob^2,
\end{equation}
where:
\begin{itemize}
\item
$\vartheta_\ell \coloneqq 2 \alpha_\ell (1 - \alpha_\ell)$ in case a step size $h_\ell = \frac{2 \alpha_\ell}{\Lip}$ with $\alpha_\ell \in (0,1)$ is used;
\item
$\vartheta_\ell \coloneqq \frac{1}{2}$ in case the step size is obtained by exact line search (if well defined);
\item
$\vartheta_\ell \coloneqq \min\left\{2 \beta \gamma (1-\gamma), \Lip \gamma \bar{h} \right\}$ in case the step size is obtained from backtracking $h_\ell = \bar{h}, \bar{h}\beta, \bar{h}\beta^2, \dots$ until the Armijo condition
\[
f(X_\ell) - f(X_{\ell+1}) \ge \gamma h_\ell \| P_{1/2} \|_\frob^2
\]
is fulfilled (here $\beta,\gamma \in (0,1)$ and $P_{1/2}$ denotes $P_1$ or $P_2$ depending on the chosen block). The backtracking procedure terminates with $h_\ell = \bar{h} \beta^i$ for some $i \le \max\left\{0, \left\lceil \ln \left( \frac{2}{\bar{h} \Lip}(1 - \gamma) \right) /  \ln \beta \right\rceil\right\}$.
\end{itemize}
\end{theorem}

\begin{corollary}\label{cor: subspace}
Under the conditions of Theorem~\ref{thm: subspace}, let $(X_\ell)$ be the iterates of Algorithm~\ref{algo: Subspace GD} using any combination of the considered step-size rules. However, if step sizes $h_\ell = \frac{2 \alpha_\ell}{\Lip}$ from the first step-size rule are used infinitely often, assume (for this subsequence) $\inf_\ell \alpha_\ell > 0$ and $\sup_\ell \alpha_\ell < 1$. Then the generated sequence $f(X_\ell)$ of function values is monotonically decreasing and converges to some value $f_* \ge \inf_{X \in \mathcal{M}_{\le k}} f(X)$. Moreover, $\mathcal P_{U_\ell,V_\ell}(\nabla f(X_\ell)) \to 0$ and $\grad f(X_\ell) \to 0$. Every accumulation point $X_*$ of the sequence $(X_\ell)$ satisfies $f(X_*) = f_*$ and $\grad f(X_*) = 0$, which means that $X_*$ is a critical point of $f$ on the manifold $\mathcal{M}_{k'}$ where $k' \coloneqq \rank(X_*) \le k$. For every nonnegative integer $j$ it holds that
\begin{equation*}\label{eq: min gradient subspace}
\min_{0 \le \ell \le j} \| \grad f(X_\ell) \|_\frob \le \min_{0\le \ell \le j} \| \mathcal{P}_{U_\ell,V_\ell}(\nabla f(X_\ell)) \|_\frob \le  \left( \frac{2\Lip}{\vartheta} \cdot \frac{f(X_0) - f_*}{j+1} \right)^{1/2},
\end{equation*}
where $\vartheta \coloneqq \inf_\ell \vartheta_\ell > 0$. In particular, given $\varepsilon > 0$ the algorithm returns a point with $\| \grad f(X_\ell) \|_\frob \le \varepsilon$ after at most $\lceil 2\Lip (f(X_0) - f_*) \vartheta^{-1} \varepsilon^{-2} - 1 \rceil$ iterations.
\end{corollary}

Note that $0 < \vartheta_\ell \le \frac{1}{2}$ in all cases of Theorem~\ref{thm: subspace} and hence $0 < \vartheta \le \frac{1}{2}$ in Corollary~\ref{cor: subspace}.

Based on the considerations above, the proofs of the theorem and the corollary are analogous to the proofs of Theorem~\ref{thm: balanced} and Corollary~\ref{cor: subspace} and are therefore omitted. There is one important detail to note. The function $X \mapsto \| \grad f(X) \|_\frob$ is not continuous but lower semicontinuous, which basically follows from the fact that the rank in a limit of matrices can only drop, and the tangent space at the limit is contained in the limit of tangent spaces. As a result, $X_\ell \to X_*$ and $\| \grad f(X_\ell) \|_\frob \to 0$ implies $\grad f(X_*) = 0$. For more discussion on limiting behavior of tangent and normal cones for low-rank matrices, see~\cite{Hosseini2019,OlikierAbsil2022}.

\subsection{Comparison of convergence guarantees}\label{sec: comparison}

We now compare the obtained convergence statements from Theorems~\ref{thm: balanced} and~\ref{thm: subspace}. For this, we first note how the norms of the Riemannian gradient and the gradient of $g(L, R) \coloneqq f(LR^\top)$ can be compared. In a balanced representation $X = L R^\top = (U \Sigma^{1/2})(V \Sigma^{1/2})^\top$ we have, by~\eqref{eq: partial gradients of g}, that
\begin{equation}\label{eq: estimate}
\begin{aligned}
\| \nabla g(L,R) \|_\frob^2 &= \| \nabla f(X) V \Sigma^{1/2} \|_\frob^2  + \| \Sigma^{1/2} U^\top \nabla f(X) \|_\frob^2\\ &\ge \sigma_{k}(X) ( \| \nabla f(X) V \|_\frob^2  + \| U^\top \nabla f(X) \|_\frob^2 ) \ge \sigma_{k}(X) \| \grad f (X) \|_\frob^2,
\end{aligned}
\end{equation}
where $\sigma_k(X)$ is the $k$th singular value of $X$. We also have a reverse estimate:
\begin{equation}\label{eq: reverse estimate}
\| \nabla g(L,R) \|_\frob^2 = \| \nabla f(X) V \Sigma^{1/2} \|_\frob^2  + \| \Sigma^{1/2} U^\top \nabla f(X) \|_\frob^2 \le 2 \| X \|_2 \| \grad f(X) \|_\frob^2.
\end{equation}

Let us take the Riemannian gradient as a reference for the comparison. For achieving a target accuracy
\[
\| \grad f(X_\ell) \|_\frob \le \varepsilon,
\]
Corollary~\ref{cor: subspace} asserts that Algorithm~\ref{algo: Subspace GD} needs at most
\[
\ell_\varepsilon \coloneqq \lceil 2\Lip (f(X_0) - f_*) \vartheta^{-1} \varepsilon^{-2} - 1 \rceil
\]
iterations. By~\eqref{eq: reverse estimate}, it implies \begin{equation*}\label{eq: implied balanced accuracy}
\| \nabla g (L_{\ell_\varepsilon},R_{\ell_\varepsilon}) \|_\frob \le \sqrt{2 \| X_{\ell_\varepsilon} \|_2} \varepsilon \le \sqrt{2 \rho} \cdot \varepsilon,
\end{equation*}
where $\rho$ is an upper bound for the spectral norm $\| X \|_2$ in the sublevel set $N_0 \coloneqq \{ X \in \mathcal{M}_{\le k} \colon f(X) \le f(L_0^{} R_0^\top) \}$. Note that under comparable assumptions (assuming the same lower bound $\vartheta \coloneqq \inf_\ell \vartheta_\ell$ and using the same $\rho$), Corollary~\ref{cor: balanced} guarantees after the same number $\ell_{\varepsilon}$ of iterations of Algorithm~\ref{algo: balanced GD} only a slightly better accuracy
\begin{equation}\label{eq: balanced accuracy}
\| \nabla g (L_\ell,R_\ell) \|_\frob \le \sqrt{\rho} \cdot \varepsilon.
\end{equation}
In this sense, Algorithm~\ref{algo: Subspace GD} automatically ensures almost the same convergence guarantee for the balanced gradient $\nabla g(L,R)$ as when using Algorithm~\ref{algo: balanced GD}.

The converse is not true, at least not from the estimates. Suppose the balanced version in Algorithm~\ref{algo: balanced GD} is used. After at most $\ell_\varepsilon$ steps it reaches the accuracy~\eqref{eq: balanced accuracy}. By~\eqref{eq: estimate}, it only implies
\[
\| \grad f(X_{\ell_\varepsilon}) \|_\frob \le \sqrt{\frac{\rho}{\sigma_{k}(X_{\ell_\varepsilon})}} \cdot \varepsilon.
\]
This estimate is not robust to the occurrence of small singular values during the iteration. In fact, when starting close enough to a potential limit point $X_*$ (of rank $k$) one obtains a prefactor roughly of the size $\sqrt{\frac{\| X_* \|_2}{ \sigma_{k}(X_*)}} = \sqrt{\frac{\sigma_1(X_*)}{\sigma_{k}(X_*)}}$, so it depends on the (relative) condition number of $X_*$. Note again that Algorithm~\ref{algo: Subspace GD} would ensure $\| \grad f(X_\ell) \|_\frob \le \varepsilon$ without any condition on singular values.

In the sense of this comparison we may conclude that the Riemannian gradient $\grad f(X_\ell)$ is a better target quantity, since it provides estimates for $\nabla g(L_\ell,R_\ell)$ independent of small singular values, which themselves are related to the local curvature of the rank-$k$ manifold; see,~e.g.,~\cite[Section~4]{FepponLermusiaux2018}. Consequently, the embedded subspace version of Algorithm~\ref{algo: Subspace GD} is a better and more stable choice for selecting alternating steepest descent directions with the Gauss--Southwell rule, since it directly aims at reducing the Riemannian gradient. The numerical experiments presented in section~\ref{sec: numerical experiments} confirm this intuition.

\subsection{Linear convergence rate}\label{sec: Linear convergence rate}

The algebraic convergence rates for the gradients obtained in Theorems~\ref{thm: balanced} and~\ref{thm: subspace} are deduced under minimal assumptions but are pessimistically slow. Instead, in practice one often observes a linear convergence rate. This is natural since the Gauss--Southwell approach leads to search directions that remain uniformly gradient-related. Indeed, in the factorized version of Algorithm~\ref{algo: balanced GD} the search directions $G_{1/2}$ satisfy the angle condition
\[
\langle G_{1/2} , \nabla g(L,R) \rangle_\frob \ge \frac{1}{\sqrt{2}} \| G_{1/2} \|_\frob \| \nabla g(L,R) \|_\frob
\]
(here $G_{1/2}$ technically is of the form $(G_1,0)$ or $(0,G_2)$) as well as the scaling condition
\[
 \frac{1}{\sqrt{2}} \| \nabla g(L,R) \|_\frob \le \| G_{1/2} \|_\frob \le \| \nabla g(L,R) \|_\frob.
\]
For the embedded Riemannian version in Algorithm~\ref{algo: Subspace GD} we have the analogous properties
\[
\langle P_{1/2} , \grad f(X) \rangle_\frob \ge \frac{1}{\sqrt{2}} \| P_{1/2} \|_\frob \| \grad f(X) \|_\frob
\]
and
\[
 \frac{1}{\sqrt{2}} \| \grad f(X) \|_\frob \le \| P_{1/2} \|_\frob \le \| \grad f(X) \|_\frob.
\]

Provided the cost function is sufficiently smooth, a descent method with such strongly gradient-related search directions can be expected to be locally linearly convergent to critical points with positive-definite Hessian. However, it is not that straightforward in our setting. For the factorized version one is faced with the problem that the cost function $g(L, R) \coloneqq f(LR^\top)$ does not have isolated critical points due to the non-uniqueness of the factorization $LR^\top$. Hence the Hessian will not be positive-definite. In our version of Algorithm~\ref{algo: balanced GD} this non-uniqueness is removed by balancing. While a local convergence analysis for Algorithm~\ref{algo: balanced GD} could likely be carried out in the corresponding quotient structure of balanced representations, we omit it here and will focus on Algorithm~\ref{algo: Subspace GD} instead.

A framework for pointwise convergence of gradient-related descent methods on conic varieties such as $\mathcal M_{\le k}$ has been developed in~\cite{Schneider2015} assuming a \L{}ojasiewicz-type inequality for the projected gradient onto tangent cones. It notably features a result for a variant of Algorithm~\ref{algo: Subspace GD} that also handles rank-deficient iterates~\cite[Algorithm~4 \& Theorem~3.10]{Schneider2015}. The required inequality for a critical point $X_* \in \mathcal M_{\le k}$ with full rank $k$ would read
\[
\abs{f(X) - f(X_*)}^{1-\theta} \le C \| \grad f(X) \|_\frob
\]
for all $X \in \mathcal M_{\le k}$ in a neighborhood of $X_*$, where $\theta \in (0,1/2]$ and $C > 0$ are constants. It is also noted in~\cite{Schneider2015} that $\theta = 1/2$ implies an asymptotic linear convergence rate, which is satisfied at smooth points when the Riemannian Hessian of $f$ at $X_*$ is positive-definite (this is proved in the lemma below; for the concept of the Riemannian Hessian we refer to~\cite[Chapter~5]{Absil2008}). Since the algorithm considered in~\cite{Schneider2015} slightly differs from Algorithm~\ref{algo: Subspace GD} and the provided convergence proof via the \L{}ojasiewicz inequality is somewhat very indirect, we present here a more direct local convergence analysis for the case of a positive-definite Riemannian Hessian. Since we already have sufficient-decrease properties as stated in Theorem~\ref{thm: subspace}---thanks to our restriction to $\Lip$-smooth cost functions---the local analysis turns out to be much easier than in the general case.

The following lemma restates the well-known fact that a \L{}ojasiewicz-type inequality with $\theta = 1/2$ holds at critical points with positive-definite (Riemannian) Hessian. 

\begin{lemma}\label{lemma: Lojasiewisc}
Let $\mathcal M$ be a smooth $d$-dimensional embedded Riemannian manifold in $\R^N$, $f \colon \mathcal M \to \R$ a twice continuously differentiable function, and $x_* \in \mathcal M$ a critical point of $f$ (i.e., $\grad f(x_*) = 0$) at which the Riemannian Hessian $H(x_*)$ is positive-definite (on $T_{x_*} \mathcal M$). Let $\lambda_{\min}^* > 0$ be its smallest eigenvalue. Then for any $\varepsilon > 0$ there exists a neighborhood $\mathcal N \subseteq \mathcal M$ of $x_*$ such that
\begin{equation*}\label{eq: gradient inequality}
0 < f(x) - f(x_*) \le \left(\frac{1}{2 \lambda_{\min}^*} + \varepsilon \right) \| \grad f(x) \|^2
\end{equation*}
for all $x \in \mathcal N$ with $x \neq x_*$.
\end{lemma}

\begin{proof}
We follow the reasoning in~\cite[Theorem~4.5.6]{Absil2008} based on local charts. Choose a local chart $\psi \colon \R^d \to \mathcal M$ for $x_*$ such that $\psi$ is a diffeomorphism between a neighborhood of zero in $\R^d$ and a neighborhood of $x_*$ in $\mathcal M$, that is, $x_* = \psi(0)$ and $f(x) = \hat f(y)$ in these neighborhoods where $\hat f \coloneqq f \circ \psi$. The function $\hat f$ satisfies $\nabla \hat f(0) = [\psi'(0)]^\top \cdot \grad f(x_*) = 0$. By replacing the diffeomorphism $\psi(y)$ with $\psi(Ty)$ if necessary, where $T$ is a suitable invertible linear transformation, we can assume that $\psi'(0) \colon \R^d \to T_{x_*} \mathcal M$ is an orthogonal map, that is, $[\psi'(0)]^\top  = [\psi'(0)]^{-1}$. Then the Hessian of $\hat f$ at zero,
\[
\hat H \coloneqq \nabla^2 \hat f(0) = [\psi'(0)]^\top \cdot  H(x_*) \cdot  [\psi'(0)],
\]
has the same eigenvalues as the Riemannian Hessian $H(x_*)$. In particular it is positive-definite so that zero is an isolated local minimizer of $\hat f$. Therefore $f(x) - f(x_*) = \hat f(y) - \hat f(0) > 0$ for $y$ close enough (but not equal) to zero. By Taylor expansion it holds that
\begin{equation}\label{eq: Taylor 2nd order}
\hat f(y) - \hat f(0) = \frac{1}{2} y^\top \hat H y + o(\| y \|^2),
\end{equation}
as well as $\nabla \hat f(y) = \hat H y + o(\| y \|)$, that is,
\begin{equation}\label{eq: y as grad fy}
y = {\hat H}^{-1} \nabla \hat f(y) + o(\|y\|).
\end{equation}
The relation~\eqref{eq: y as grad fy} in particular implies that for given $0 < \varepsilon' < 1$ and $\| y \|$ small enough one has
\[
(1 - \varepsilon') \| y \| \le  \|{\hat H}^{-1} \nabla \hat f(y) \|  \le (1 + \varepsilon') \|y \|.
\]
Using this when inserting~\eqref{eq: y as grad fy} into~\eqref{eq: Taylor 2nd order} allows to conclude that
\[
\hat f(y) - \hat f(0) 
= \frac{1}{2}\nabla \hat f(y)^\top \hat H^{-1} \nabla \hat f(y) + o( \| \nabla \hat f(y) \|^2) \le \frac{1}{2\lambda_{\min}^*} \| \nabla \hat f(y) \|^2 + o( \| \nabla \hat f(y) \|^2).
\]
Since $\psi'(0)$ is orthogonal, we have $\| \nabla \hat f(y) \| = \| [\psi'(y)]^\top \grad f(x) \| \le (1 + \varepsilon'') \| \grad f(x) \|$ for any given $\varepsilon'' > 0$ provided $\| y \|$ is small enough. This implies the assertion.
\end{proof}

We now state a linear convergence result for Algorithm~\ref{algo: Subspace GD} for $\Lip$-smooth functions under the assumption that one of the accumulation points has full rank $k$ and a positive-definite Riemannian Hessian with respect to the smooth embedded submanifold
\[
\mathcal M_k \coloneqq \{ X \in \R^{m \times n} \colon \rank (X) = k \}.
\]
The result applies to any combination of step-size rules considered in Theorem~\ref{thm: subspace}, but requires additional assumptions for fixed step sizes (as before) and for exact step sizes in order to ensure suitable lower (in case of the former) and upper (for the latter) bounds on the step sizes.

\begin{theorem}\label{thm: linear rate}
Assume that $f \colon \R^{m \times n} \to \R$ is $\Lip$-smooth and bounded from below. Let $(X_\ell)$ be the iterates of Algorithm~\ref{algo: Subspace GD} using any combination of the step-size rules considered in Theorem~\ref{thm: subspace}. If step sizes $h_\ell = \frac{2 \alpha_\ell}{\Lip}$ are used for infinitely many iterates, assume (for this subsequence) $\inf_\ell \alpha_\ell > 0$ and $\sup_\ell \alpha_\ell < 1$, implying $\vartheta \coloneqq \inf_\ell \vartheta_\ell > 0$ for the $\vartheta_\ell$ defined in Theorem~\ref{thm: subspace}. If exact line search is used infinitely often, assume in addition that the constrained sublevel set $N_0 \coloneqq \{X \in \mathcal{M}_{\le k} : f(X) \le f(X_0) \}$ is bounded.

Suppose that $(X_\ell)$ has an accumulation point $X_*$ of rank $k$ such that $f$ is twice continuously differentiable in a neighborhood of $X_*$ and the Riemannian Hessian $H(X_*)$ of $f \vert_{\mathcal M_k}$ at $X_*$ is positive-definite. Then $X_*$ is the unique limit point of $(X_\ell)$, an isolated local minimizer of $f \vert_{\mathcal M_k}$, and the convergence is R-linear with a rate
\[
\limsup_{\ell \to \infty} \| X_\ell - X_* \|^{1/\ell} \le \left( 1 - \vartheta \frac{\lambda_{\min}^*}{\Lip} \right)^{1/2} < 1,
\]
where $\lambda_{\min}^*$ is the smallest eigenvalue of $H(X_*)$. For the function values it holds that
\[
\limsup_{\ell \to \infty} \, [ f(X_\ell) - f(X_*) ]^{1/\ell} \le 1 - \vartheta \frac{\lambda_{\min}^*}{\Lip}.
\]
\end{theorem}

Before giving the proof, we note that the asserted inequality $\lambda^*_{\min} \le \lambda$ holds for all Riemannian manifolds that are invariant under multiplication by positive real numbers, because the restriction of the Riemannian Hessian to one-dimensional subspaces corresponding to rays within the manifold just equals the restriction of the Euclidean Hessian to that subspace. Hence $\lambda^*_{\min}$ is bounded by the largest eigenvalue of the Euclidean Hessian, which in turn is bounded by the Lipschitz constant $\Lip$ of the gradient.

In fact, since in the Gauss--Southwell rule we only deal with linear subspaces within the tangent spaces of the low-rank manifold, a somewhat similar logic allows to obtain a refinement of Theorem~\ref{thm: linear rate}, which applies when only Armijo backtracking or exact line search are used infinitely often. The proofs of both theorems will be given subsequently.

\begin{theorem}\label{thm: linear rate 2}
Under the conditions of Theorem~\ref{thm: linear rate} assume that, except for finitely many steps, exact line search or Armijo backtracking is used. Let $\nabla^2 f(X_*)$ be the Euclidean Hessian of $f$ at $X_*$. Then
\[
\limsup_{\ell \to \infty} \| X_\ell - X_* \|^{1/\ell} \le \left( 1 - \vartheta^* \frac{\lambda_{\min}^*}{\lambda^*} \right)^{1/2} < 1,
\]
where $\lambda^* \le \lambda$ is the maximum between the operator norms (induced by the Frobenius norm) of the two projected Hessians $\mathcal P_{V_*} \nabla^2 f(X_*) \mathcal P_{V_*}$ and $\mathcal P_{U_*} \nabla^2 f(X_*) \mathcal P_{U_*}$, and 
\[
\vartheta^* \coloneqq \min \{2 \beta \gamma (1-\gamma),\lambda^* \gamma \bar{h} \}.
\]
For the function values it holds that
\[
\limsup_{\ell \to \infty} \, [f(X_\ell) - f(X_*)]^{1/\ell} \le 1 - \vartheta^* \frac{\lambda_{\min}^*}{\lambda^*}.
\]
If only exact line search is used (up to finitely many steps), the above estimates hold with
\[
\vartheta^* = \frac{1}{2}.
\]
\end{theorem}

Note that $\lambda^*$ can be bounded from above by more natural quantities such as the largest eigenvalue of $\mathcal P_{U_*,V_*} \nabla^2 f(X_*) \mathcal P_{U_*,V_*}$ (i.e.,~the restriction of $\nabla^2 f(X_*)$ to the tangent space $T_{X_*} \mathcal M_k$), the largest eigenvalue of the Riemannian Hessian of $f\vert_{\mathcal M_k}$ at $X_*$, or simply the largest eigenvalue of $\nabla^2 f(X_*)$.

\begin{proof}[Proof of Theorem~\ref{thm: linear rate}]
As in the proof of Lemma~\ref{lemma: Lojasiewisc} we consider a local chart $X = \psi(y)$ for $\mathcal M_k$ at $X_*$, that is, $\psi$ is a diffeomorphism between a neighborhood of zero in $\R^d$ and a neighborhood of $\psi(0) = X_*$ in $\mathcal M_k$. Let $\hat H$ be the Hessian of $f \circ \psi$ at zero, which is positive-definite. The expression $\| y \|_{\hat H} \coloneqq (\frac{1}{2} y^\top \hat H y)^{1/2}$ hence defines a norm on $\R^d$. We can assume that $\psi$ is a diffeomorphism on an open ball $\| y \|_{\hat H} < r$ around zero. Let $\varepsilon > 0$. By the lemma, there exists $r_\varepsilon \in (0, r]$ such that
\begin{equation}\label{eq: grad inequality}
f(\psi(y)) - f(X_*) \le \left( \frac{1}{2 \lambda_{\min}^*} + \varepsilon \right) \| \grad f(\psi(y)) \|^2_\frob \quad \text{for all $\| y \|_{\hat H} < r_\varepsilon$.}
\end{equation}
Let $\delta > 0$. Then the Taylor expansion as in~\eqref{eq: Taylor 2nd order} implies that there is $r_{\delta} \in (0, r_{\varepsilon}]$ such that
\[
(1 - \delta) \| y \|_{\hat H}^2 \le \abs{f(\psi(y)) - f(\psi(0))} \le (1 + \delta) \| y \|_{\hat H}^2 \quad \text{for all $\| y \|_{\hat H} < r_{\delta}$.}
\]

The iterates of Algorithm~\ref{algo: Subspace GD} satisfy
\begin{equation}\label{eq: distance of iterates}
\| X_{\ell+1} - X_{\ell}\|_\frob = h_\ell  \| P_{1/2} \|_\frob \le h_\ell \| \grad f(X_\ell) \|_\frob.
\end{equation}
Hence $X_{\ell+1} - X_{\ell} \to 0$ since $\grad f(X_\ell) \to 0$ by Corollary~\ref{cor: subspace} and the $h_\ell$ are bounded (for the exact line search one uses here that $N_0$ is bounded). This implies that there exist $\ell_0 \in \mathbb N$ and $r'_{\delta} \in (0, r_\delta]$ such that for all $\ell \ge \ell_0$ the property $X_{\ell} = \psi(y_\ell)$ with $\| y_\ell \|_{\hat H} < r'_{\delta}$ implies $X_{\ell + 1} = \psi(y_{\ell+1})$ with $\| y_{\ell+1} \|_{\hat H} < r_{\delta}$, and hence both inequalities
\begin{equation}\label{eq: function values 1}
f(X_\ell) - f(X_*) \le (1 + \delta) \| y_\ell \|_{\hat H}^2
\end{equation}
and
\begin{equation}\label{eq: function values 2}
f(X_{\ell+1}) - f(X_*) \ge (1 - \delta) \| y_{\ell+1} \|_{\hat H}^2
\end{equation}
are applicable, where we use $f(X_\ell) \ge f(X_{\ell+1}) \ge f(X_*)$ by construction. Since $X_*$ is an accumulation point, such $\ell$ exist. Consider one such $\ell$. By~\eqref{eq: grad inequality}, it satisfies
\begin{equation}\label{eq: grad inequality matrix}
f(X_\ell) - f(X_*) \le \left( \frac{1}{2 \lambda_{\min}^*} + \varepsilon \right) \| \grad f(X_\ell) \|^2.
\end{equation}
Combining this inequality with the main estimates~\eqref{eq: general estimate subspace} in Theorem~\ref{thm: subspace} results in
\begin{align*}\label{eq: contraction function values}
f(X_{\ell+1}) - f(X_*) &\le f(X_\ell) - f(X_*) - \frac{\vartheta_\ell}{2 \Lip} \| \grad f(X_\ell) \|^2 \notag \\
&\le \left( 1 - \frac{\vartheta_\ell}{\frac{\Lip}{\lambda_{\min}^*} + 2 \Lip \varepsilon} \right) (f(X_\ell) - f(X_*)).
\end{align*}
From~\eqref{eq: function values 1} and~\eqref{eq: function values 2} we conclude
\[
\| y_{\ell+1} \|_{\hat H}^2 \le \left( \frac{1 + \delta}{1 - \delta} \right)\left( 1 - \frac{\vartheta_\ell}{\frac{\Lip}{\lambda_{\min}^*} + 2 \Lip \varepsilon} \right) \| y_\ell \|_{\hat H}^2.
\]
Since $\vartheta \coloneqq \inf_\ell \vartheta_\ell > 0$, we can now assume $\ell_0$ is large enough and $\delta$ is small enough such that the product of both parentheses is bounded by some value $q < 1$ for all $\ell \ge \ell_0$.  For the particular $\ell$ under consideration it then follows $\| y_{\ell+1} \|_{\hat H} \le \sqrt{q}\| y _\ell\|_{\hat H} < r'_{\delta}$, so that the argument can be inductively repeated. This implies $y_\ell \to 0$ and hence $X_\ell \to X_*$. The estimates above then imply
\begin{equation}\label{eq: limsup1}
\limsup_{\ell \to \infty} \, (f(X_\ell) - f(X_*))^{1/\ell} \le 1 - \frac{\vartheta}{\frac{\Lip}{\lambda_{\min}^*} + 2 \Lip \varepsilon}
\end{equation}
and
\begin{equation}\label{eq: limsup2}
\limsup_{\ell \ge \ell_0} \| y_\ell \|^{1/\ell} \le \sqrt{\left( \frac{1 + \delta}{1 - \delta} \right)\left( 1 - \frac{\vartheta}{\frac{\Lip}{\lambda_{\min}^*} + 2 \Lip \varepsilon} \right)}.
\end{equation}
Finally, note that
\begin{align*}
\limsup_{\ell \to \infty} \| X_\ell - X_* \|_\frob^{1/\ell} &= \limsup_{\ell \ge \ell_0} \| \psi(y_\ell) - \psi(0)  \|_\frob^{1/\ell} \le \limsup_{\ell \to \infty} \| y_\ell  \|_\frob^{1/\ell}
\end{align*}
since $\psi$ is locally Lipschitz continuous. As $\delta$ and $\varepsilon$ can be taken arbitrarily small, this proves the theorem.
\end{proof}

\begin{proof}[Proof of Theorem~\ref{thm: linear rate 2}]
We have already shown that $X_\ell$ converges to $X_*$. Moreover, $\grad f(X_\ell) \to 0$ by Theorem~\ref{thm: subspace}. We can now refer to the remark at the end of section~\ref{sec: gradient descent methods}: assuming $X_{\ell+1}$ is obtained from $X_\ell$ using the Armijo backtracking, then in the estimate in~\eqref{eq: general estimate subspace} we can replace $\Lip$ with a restricted smoothness constant $\lambda_\ell$ of the projected gradient on the line segment connecting $X_\ell$ and $X_\ell - \bar{h} P_{1/2,\ell}$ as in~\eqref{eq: restricted smoothness}. For concreteness, assume that $P_{1,\ell}$ is used. Let $\phi$ be the restriction of $f$ to $T_{V_\ell}$. Then $\nabla \phi(X) = \mathcal P_{V_\ell} \nabla f(X)$ and $\nabla^2 \phi(X) = \mathcal P_{V_\ell} \nabla^2 f(X) \mathcal P_{V_\ell}$. Let $0 \le t \le 1$. Using a mean value inequality for the vector-valued function $\xi \mapsto \nabla \phi(X_\ell)  - \nabla \phi(X_\ell - \xi t \bar{h} P_{1,\ell})$, $\xi \in [0,1]$, gives that
\begin{align*}
\| \nabla \phi(X_\ell)  - \nabla \phi(X_\ell - t \bar{h} P_{1,\ell}) \|_\frob &\le \max_{\xi \in [0,1]} \| \nabla^2 \phi(X_\ell - \xi t \bar{h} P_{1,\ell}) \cdot t \bar{h} P_{1,\ell} \|_{\frob} \\
&\le \max_{\xi \in [0,1]} \| \nabla^2 \phi(X_\ell - \xi \bar{h} P_{1,\ell}) \|_{\frob \to \frob} \cdot t \bar{h} \| P_{1,\ell} \|_\frob.
\end{align*}
Here, $\|\cdot\|_{\mathrm{F}\to\mathrm{F}}$ is the operator norm induced by the Frobenius norm. Note that since $X_\ell \to X_*$ and $P_{1,\ell} \to 0$, the function $\phi$ is indeed differentiable on the line segment in question for $\ell$ large enough by assumption. Referring to~\eqref{eq: restricted smoothness}, it means that in that step we can replace $\lambda$ in~\eqref{eq: general estimate subspace} with
\[
\lambda_\ell \coloneqq \max_{\xi \in [0,1]} \| \nabla^2 \phi(X_\ell - \xi \bar{h} P_{1,\ell}) \|_{\frob \to \frob} = \max_{\xi \in [0,1]} \| \mathcal P_{V_\ell} \nabla^2 f(X_\ell - \xi \bar{h} P_{1,\ell}) \mathcal P_{V_\ell} \|_{\frob \to \frob}.
\]
If $P_{2,\ell}$ is used, we take
\[
\lambda_\ell \coloneqq \max_{\xi \in [0,1]} \| \mathcal P_{U_\ell} \nabla^2 f(X_\ell - \xi \bar{h} P_{2,\ell}) \mathcal P_{U_\ell} \|_{\frob \to \frob}
\]
instead. Let $\eta > 0$. By continuity, we conclude from $X_\ell \to X_*$ and $P_{1/2,\ell} \to 0$ that for $\ell$ large enough and in case the Armijo step size is used we can replace~\eqref{eq: general estimate subspace} with
\begin{equation}\label{eq: refined estimate}
f(X_\ell) - f(X_{\ell+1})  \ge \frac{\vartheta_\eta}{2 \lambda_\eta} \| \grad f(X_\ell) \|_\frob^2,
\end{equation}
where
\begin{align*}
\lambda_\eta &\coloneqq (1+\eta) \max\left\{ \| \mathcal P_{V_*} \nabla^2 f(X_*) \mathcal P_{V_*} \|_{\frob \to \frob}, \| \mathcal P_{U_*} \nabla^2 f(X_*) \mathcal P_{U_*} \|_{\frob \to \frob} \right\} \\ &= (1 + \eta) \lambda^*
\end{align*}
and
\[
\vartheta_\eta \coloneqq \min\left\{2 \beta \gamma (1-\gamma), \lambda_\eta \gamma \bar{h} \right\}.
\]
Obviously, for exact line search the same estimate~\eqref{eq: refined estimate} must then also hold. If it is combined with the gradient inequality~\eqref{eq: grad inequality matrix} as in the proof of Theorem~\ref{thm: linear rate}, one arrives at estimates like~\eqref{eq: limsup1} and~\eqref{eq: limsup2}, but with $\Lip$ replaced with $\lambda_\eta$ and $\vartheta$ replaced with $\vartheta_\eta$, which leads to the improved rates, since $\eta$ can be taken arbitrarily small.

If only exact line search is used for $\ell$ large enough, then the Armijo parameters in the above analysis are only auxiliary quantities. We can then optimize $\vartheta_\eta$ by choosing $\gamma \coloneqq \frac{1}{2}$, $\beta$ arbitrarily close to one, and $\bar{h}$ large enough, so that $\vartheta_\eta$ is arbitrarily close to $\frac{1}{2}$.
\end{proof}

\subsection{Alternating least squares}\label{sec: ALS}

Interestingly, our results on the gradient-related algorithms based on Gauss--Southwell rule allow to make some conclusions on another popular algorithm used in low-rank optimization, the \emph{alternating least squares} (ALS) algorithm, in the context of $\Lip$-smooth functions. Ignoring details on implementation, this algorithm operates as follows.

\begin{algorithm}[H]
\small
\caption{Alternating least squares (ALS)}
\label{algo: ALS}
\begin{algorithmic}[1]
\Input
Low-rank decomposition $X_0^{} = U_0^{} V_0^\top$ with $U_0 \in \R^{m \times k}$, $V_0 \in \R^{n \times k}$
\For
{$\ell = 0,1,2,\dots$}
\State Replace $V_\ell$ with a QR-factor: $V_{\ell} \leftarrow \mathrm{qf}(V_{\ell})$.
\State Find
\[
U_{\ell+1} \in \argmin_{U \in \R^{m \times k}} f(U V_\ell^\top).
\]
\State Replace $U_{\ell+1}$ with a QR-factor: $ U_{\ell+1} \leftarrow \mathrm{qf}(U_{\ell+1})$.
\State Find
\[
V_{\ell+1} \in \argmin_{V \in \R^{n \times k}} f(U_{\ell+1} V^\top).
\]
\EndFor
\end{algorithmic}
\end{algorithm}

The name alternating least squares stems from the fact that this algorithm is usually applied to convex quadratic cost functions $f$ such as $L_2$-losses, which makes it possible to solve the minimization problems in the substeps via least squares problems. Note that the orthogonalization steps in lines 2 and 4 are not needed from a theoretical perspective, but ensure the numerical stability of the subproblems in practice.

Let $X_\ell = U_{\ell}^{} V_\ell^\top$. There are two simple but crucial observations for relating the ALS algorithm to our previous results. The first is that by construction it trivially follows a Gauss--Southwell rule for selecting the next block variable to be updated. In fact, before the first half-step, which finds an update $X_{\ell + 1/2} = U_{\ell + 1}^{} V_\ell^\top$ in the subspace~$T_{V_\ell}$, the projected gradient
\(
P_{2,\ell} = \mathcal P_{U_\ell} \nabla f(X_{\ell})
\)
onto $T_{U_\ell}$ is just zero (for $\ell \ge 1$), since by the result of the previous step $X_\ell$ is the global minimizer on $T_{U_\ell}$. Hence the Gauss--Southwell rule (trivially) also would select the subspace $T_{V_\ell}$ for the next update, and a corresponding gradient step in $T_{V_\ell}$ would satisfy the estimates in~\eqref{eq: general estimate subspace}, but even without the factor $1/2$ since $\| P_{2,\ell} \|_\frob = \| \grad f(X_\ell) \|_\frob$.

The second observation is that the decrease in function value for the ALS half-step is better than in any of the gradient methods, since $X_{\ell + 1/2}$ is the global minimizer on~$T_{V_\ell}$. Hence the best of the estimates~\eqref{eq: general estimate subspace}, which is the one with the exact line search, is valid for the ALS half-step. Analogous considerations apply to the second half-step. Relabelling the iterates of Algorithm~\ref{algo: ALS} such that
\[
\tilde X_{2\ell} = X_\ell, \quad \tilde X_{2\ell + 1} = X_{\ell+1/2},
\]
we obtain from these considerations that
\begin{equation}\label{eq: estimate ALS}
f(\tilde X_\ell) - f(\tilde X_{\ell+1}) \ge \frac{1}{2\lambda} \| \grad f(\tilde X_\ell) \|_\frob^2,
\end{equation}
that is, according to the previous notation we have $\vartheta_\ell = \vartheta = 1$ for all half-steps. In a sense, this approach of obtaining estimates by comparison to a gradient method can be seen as an instance of the accelerated line-search methods as considered in~\cite{Absil2009}.

The estimate of the form~\eqref{eq: estimate ALS} is the basis for the statements on the gradient convergence such as in Corollary~\ref{cor: subspace}. This allows to formulate the corresponding conclusion for the sequence $\tilde X_\ell$, essentially by using $\vartheta = 1$ in Corollary~\ref{cor: subspace}. If we switch to the original ALS sequence $X_\ell = \tilde X_{2\ell}$ of Algorithm~\ref{algo: ALS} consisting of two half-steps, the following result is obtained.

\begin{corollary}\label{cor: ALS}
Assume that $f \colon \R^{m \times n} \to \R$ is $\Lip$-smooth and bounded from below. Let $(X_\ell^{}) = (U_\ell^{} V_\ell^\top)$ be the iterates of Algorithm~\ref{algo: ALS}.  Then the generated sequence $f(X_\ell)$ of function values is monotonically decreasing and converges to some value $f_* \ge \inf_{X \in \mathcal{M}_{\le k}} f(X)$. Moreover, $\mathcal P_{U_\ell,V_\ell}(\nabla f(X_\ell)) \to 0$ and $\grad f(X_\ell) \to 0$. Every accumulation point $X_*$ of the sequence $(X_\ell)$ satisfies $f(X_*) = f_*$ and $\grad f(X_*) = 0$, which means that $X_*$ is a critical point of $f$ on the manifold $\mathcal{M}_{k'}$ where $k' \coloneqq \rank(X_*) \le k$. For every nonnegative integer $j$ it holds that
\begin{equation*}
\min \{ \| \grad f(X_\ell) \|_\frob, \, \| \grad f(X_{\ell+1/2}) \|_\frob \colon 0 \le \ell \le j \} \le \left( 2\Lip \cdot \frac{f(X_0) - f_*}{2j+1} \right)^{1/2}.
\end{equation*}
In particular, given $\varepsilon > 0$ the algorithm returns a point satisfying $\| \grad f(X_\ell) \|_\frob \le \varepsilon$ or $\| \grad f(X_{\ell+1/2}) \|_\frob \le \varepsilon$ after at most $ \lceil \Lip (f(X_0) - f_*)\varepsilon^{-2} - \frac{1}{2} \rceil$ iterations.
\end{corollary}

It is interesting to note that~\eqref{eq: estimate ALS} can be applied in a slightly different way, as it implies the two separate estimates
\[
f(X_\ell) - f(X_{\ell+1}) \ge f(X_{\ell+1/2}) - f(X_{\ell+1}) \ge \| \grad f(X_{\ell+1/2}) \|_\frob^2
\]
and
\[
f(X_\ell) - f(X_{\ell+1}) \ge f(X_\ell) - f(X_{\ell+1/2}) \ge \| \grad f(X_\ell) \|_\frob^2,
\]
which under the assumptions of Corollary~\ref{cor: ALS} lead to the statement
\[
\min_{0 \le \ell \le j} \max \big\{ \| \grad f(X_\ell) \|_\frob,\| \grad f(X_{\ell+1/2}) \|_\frob \big\}
\le  \left( 2\Lip \cdot \frac{f(X_0) - f_*}{j+1} \right)^{1/2}.
\]
In this form the result is quite analogous to~\cite[Theorem~3.2]{Beck2015}, which (as a special case) provides global rates for the (Euclidean) gradients in the alternating minimization method for non-convex functions with two block variables. In our case, this would be the function $g(L,R)$, however the assumptions in~\cite{Beck2015} require that the partial gradients (or at least one of them) satisfy uniform Lipschitz conditions, which  by~\eqref{eq: Lipschitz for L} and~\eqref{eq: Lipschitz for R} cannot be guaranteed due to the non-uniqueness of the $LR$ factorization. We fixed this by replacing the partial gradients with their scaled or Riemannian counterparts. Apart from that, the reasoning of~\cite{Beck2015} is actually similar to ours, as it is also based on comparing the single steps in alternating minimization to partial gradient descent. Finally note that as in section~\ref{sec: embedded version} the Riemannian gradients in our results for ALS could in fact be replaced by the projection of $\nabla f(X_\ell)$ onto $T_{U_\ell,V_\ell}$, which in theory provides slightly more precise information in case of rank-deficient iterates or accumulation points.

The next natural step is to apply the local linear convergence analysis as conducted for Theorems~\ref{thm: linear rate} and~\ref{thm: linear rate 2} to ALS. Here one principal difficulty is to ensure $X_{\ell+1} - X_\ell \to 0$, which is a crucial ingredient in the proof of Theorem~\ref{thm: linear rate}. For the gradient-related search directions as considered there this follows from $\grad f(X_\ell) \to 0$ together with~\eqref{eq: distance of iterates}, but for ALS it seems to require additional assumptions (or modifications of the algorithm). 
One such assumption is that the cost function $f$ is strongly convex, that is, there exists $\mu > 0$ such that
\[
f(Y) \ge f(X) + \langle \nabla f(X), Y- X \rangle_{\frob} + \frac{\mu}{2} \| Y - X \|_{\frob}^2
\]
for all $X,Y \in \R^{m \times n}$. This still covers a large field of applications.

\begin{theorem}
Assume that $f \colon \R^{m \times n} \to \R$ is $\Lip$-smooth and strongly convex. Let $(X_\ell) = (U_\ell^{} V_\ell^\top)$ be the iterates of Algorithm~\ref{algo: ALS}. Assume that $(X_\ell)$ has an accumulation point $X_*$ with full rank $k$ such that $f$ is twice continuously differentiable in a neighborhood of $X_*$ and the Riemannian Hessian $H(X_*)$ of $f \vert_{\mathcal M_k}$ at $X_*$ is positive-definite. Then $X_*$ is the unique limit point of $(X_\ell)$, an isolated local minimizer of $f \vert_{\mathcal M_k}$, and the convergence is R-linear with a rate
\[
\limsup_{\ell \to \infty} \| X_\ell - X_* \|^{1/\ell} \le 1 - \frac{\lambda_{\min}^*}{\lambda^*}  < 1,
\]
where $\lambda_{\min}^*$ is the smallest eigenvalue of $H(X_*)$ and $\lambda^*$ is the maximum between the operator norms (induced by the Frobenius norm) of the two projected Hessians $\mathcal P_{V_*} \nabla^2 f(X_*) \mathcal P_{V_*}$ and $\mathcal P_{U_*} \nabla^2 f(X_*) \mathcal P_{U_*}$. For the function values it holds that
\[
\limsup_{\ell \to \infty} \, [f(X_\ell) - f(X_*)]^{1/\ell} \le \left( 1 -  \frac{\lambda_{\min}^*}{\lambda^*} \right)^2.
\]
\end{theorem}

We note again that $\lambda^*$ can be bounded from above by the largest eigenvalue of $\mathcal P_{U_*,V_*} \nabla f(X_*) \mathcal P_{U_*,V_*}$, the largest eigenvalue of the Riemannian Hessian of $f\vert_{\mathcal M_k}$ at $X_*$, or simply the largest eigenvalue of $\nabla^2 f(X_*)$.

\begin{proof}
Instead of $(X_\ell)$ we consider as above the sequence $(\tilde X_\ell)$ of half-steps. It satisfies
\[
f(\tilde X_\ell) - f(\tilde X_{\ell+1}) \ge \langle \nabla f(\tilde X_{\ell+1}), \tilde X_\ell - \tilde X_{\ell+1} \rangle_\frob + \frac{\mu}{2} \| \tilde X_\ell - \tilde X_{\ell+1} \|_\frob^2 = \frac{\mu}{2} \| \tilde X_\ell - \tilde X_{\ell+1} \|_\frob^2,
\]
since by construction $\tilde X_{\ell+1}$ minimizes $f$ on a linear subspace that contains $\tilde X_\ell - \tilde X_{\ell+1}$ (hence $\langle \nabla f(\tilde X_{\ell+1}), \tilde X_\ell - \tilde X_{\ell+1} \rangle_\frob = 0$). The sequence $f(\tilde X_\ell)$ is monotonically decreasing and converges to some $f_*$ by Corollary~\ref{cor: ALS}. In particular, $f(\tilde X_\ell) - f(\tilde X_{\ell+1}) \to 0$, which by the above inequality implies $\| \tilde X_\ell - \tilde X_{\ell+1} \|_\mathrm{F} \to 0$. From here on one can first repeat the arguments in the proof of Theorem~\ref{thm: linear rate} for the sequence $(\tilde X_\ell)$ (one has $\vartheta_\ell = 1$ by~\eqref{eq: estimate ALS}), which in particular prove that $X_*$ is the unique limit point of the sequence.

One can then proceed precisely as in the proof of Theorem~\ref{thm: linear rate 2} for obtaining the refined estimate
\[
 f(\tilde X_\ell) - f(\tilde X_{\ell+1}) \ge \frac{\vartheta_\eta}{\lambda_\eta} \| \grad f(\tilde X_\ell) \|^2_\frob
\]
similar to~\eqref{eq: refined estimate} with a local smoothness constant, which holds for $\ell$ large enough for a projected gradient step from $\tilde X_\ell$ when using the Armijo backtracking, and hence by construction it also holds for the ALS substep as it maximizes the left-hand side in the corresponding subspace. Note that the above estimate differs from~\eqref{eq: refined estimate} by the factor $1/2$, which again accounts for the fact that $\| P_{1/2,\ell} \|_{\frob} = \| \grad f(\tilde X_\ell) \|_\frob$ by the properties of ALS. Since $\lambda_\eta$ is arbitrarily close to $\lambda^*$ and $\vartheta_\eta$ can be chosen arbitrarily close to $1/2$ using suitable Armijo parameters, one obtains as in the aforementioned proofs (essentially replacing $\vartheta_\eta$ with $2 \vartheta_\eta \approx 1$) the R-linear rates
\(
\limsup_{\ell \to \infty} \| \tilde X_\ell - X_* \|^{1/\ell} \le \left( 1 - \frac{\lambda_{\min}^*}{\lambda^*} \right)^{1/2}
\)
and
\(
\limsup_{\ell \to \infty} \, [f(\tilde X_\ell) - f(X_*)]^{1/\ell} \le  1 -  \frac{\lambda_{\min}^*}{\lambda^*}.
\)
Noting that $X_\ell = \tilde X_{2\ell}$ gives the result.
\end{proof}

\section{Numerical experiments}\label{sec: numerical experiments}

We present the results of some numerical experiments that highlight the differences between Algorithms~\ref{algo: balanced GD} and~\ref{algo: Subspace GD}. As objective function, we take the strongly convex quadratic function 
\[
 f(X) \coloneqq \tfrac{1}{2} \tr( X^\top \mathcal{A}(X) ) -    \tr( X^\top B )
\]
with a prescribed condition number $\kappa_{\mathcal A}$. 
The linear, symmetric positive-definite operator $\mathcal{A}\colon \mathbb{R}^{n \times n} \to \mathbb{R}^{n \times n}$ is random with $n$ logarithmically spaced eigenvalues between $\lambda_{\min} = 1/\kappa_{\mathcal A} > 0$ and $\lambda_{\max} = 1$. The unique minimizer $X_*$ of $f$ is constructed randomly but with a prescribed rank $k$ and prescribed effective condition number $\kappa_{X_*} = \sigma_1(X_*) / \sigma_k(X_*)$. Its singular values are also logarithmically spaced. The matrix $B$ is determined from $\nabla f(X^ *) = \mathcal{A}(X_*) - B = 0$. Since $f$ is quadratic, an exact line search can be easily calculated in closed form. Other line searches will be tested and explained further below.

It has been reported in a number of works~\cite{Tanner2016,TongMaChi2021,Vandereycken2013} that ALS (and ALS-like methods that solve the subproblems in steps 3 and 5 of Algorithm~\ref{algo: ALS} exactly) are often not competitive compared with (Riemannian) gradient methods like Algorithms~\ref{algo: balanced GD} and~\ref{algo: Subspace GD}. A general observation is that while ALS converges in less steps than gradient methods, the computational cost for solving these subproblems is not offset when the rank is sufficiently large. A precise comparison would be very problem dependent though. We will not compare the computational costs of Algorithms~\ref{algo: balanced GD} and~\ref{algo: Subspace GD} with that of Algorithm~\ref{algo: ALS}.

We first remark that Algorithms~\ref{algo: balanced GD} and \ref{algo: Subspace GD} have similar computational costs. Being gradient-based methods, they both require the computation of $\nabla f(X)$ for a rank-$k$ matrix $X$. As other dominant cost per step, both algorithms require $k$ matrix-vector products with $\nabla f$ and $(\nabla f)^\top$. In large-scale applications, it is important to exploit the low-rank structures\footnote{Consider for example the case of a gradient that has low operator rank: $\nabla f(X) = \sum_{i=1}^ \ell A_i X B_i$. If $X=GH^\top$, an efficient evaluation of $\nabla f(X) G$ is $\sum_{i=1}^ \ell (A_i G) (H^\top B_i G)$.} since otherwise these computations will dominate all others. Next, both algorithms compute two compact QR factorizations of matrices of same size. Since Algorithm~\ref{algo: Subspace GD} does not perform explicit balancing based on a small $k \times k$ SVD, it is slightly faster per step. As we will see in the numerical experiment later, Algorithm~\ref{algo: Subspace GD} also converges faster for ill-conditioned problems.

\vspace*{-1ex}
\paragraph{Influence of balancing and orthogonalization.} We report on the behavior of Algorithms~\ref{algo: balanced GD} and~\ref{algo: Subspace GD} for $n=30$ and rank $k=8$ in the above setting. Each algorithm was started in the same set of 100 random initial points $X_0$. We also included the results for a variant of Algorithm~\ref{algo: balanced GD} without balancing, where line 4 in Algorithm~\ref{algo: balanced GD} simply reads $L_{\ell+1} = L_+$ and $R_{\ell+1} = R_+$.

In Figure~\ref{fig:influence balancing}, we show the results for $\kappa_{\mathcal A}=20$ and $\kappa_{X_*} = 5$, and with exact line search. It is immediately clear that balancing (Algorithm~\ref{algo: balanced GD}) or orthogonalization (Algorithm~\ref{algo: Subspace GD}) leads to a substantial improvement in convergence speeds. In addition, Algorithm~\ref{algo: Subspace GD}, which uses orthogonalization (or can be seen as a Riemannian block method), is considerably faster and has less variance than Algorithm~\ref{algo: balanced GD}, which only balances.

\begin{figure}[t!]
    \centering
    \begin{minipage}{.5\textwidth}
        \centering
        \includegraphics[width=\linewidth]{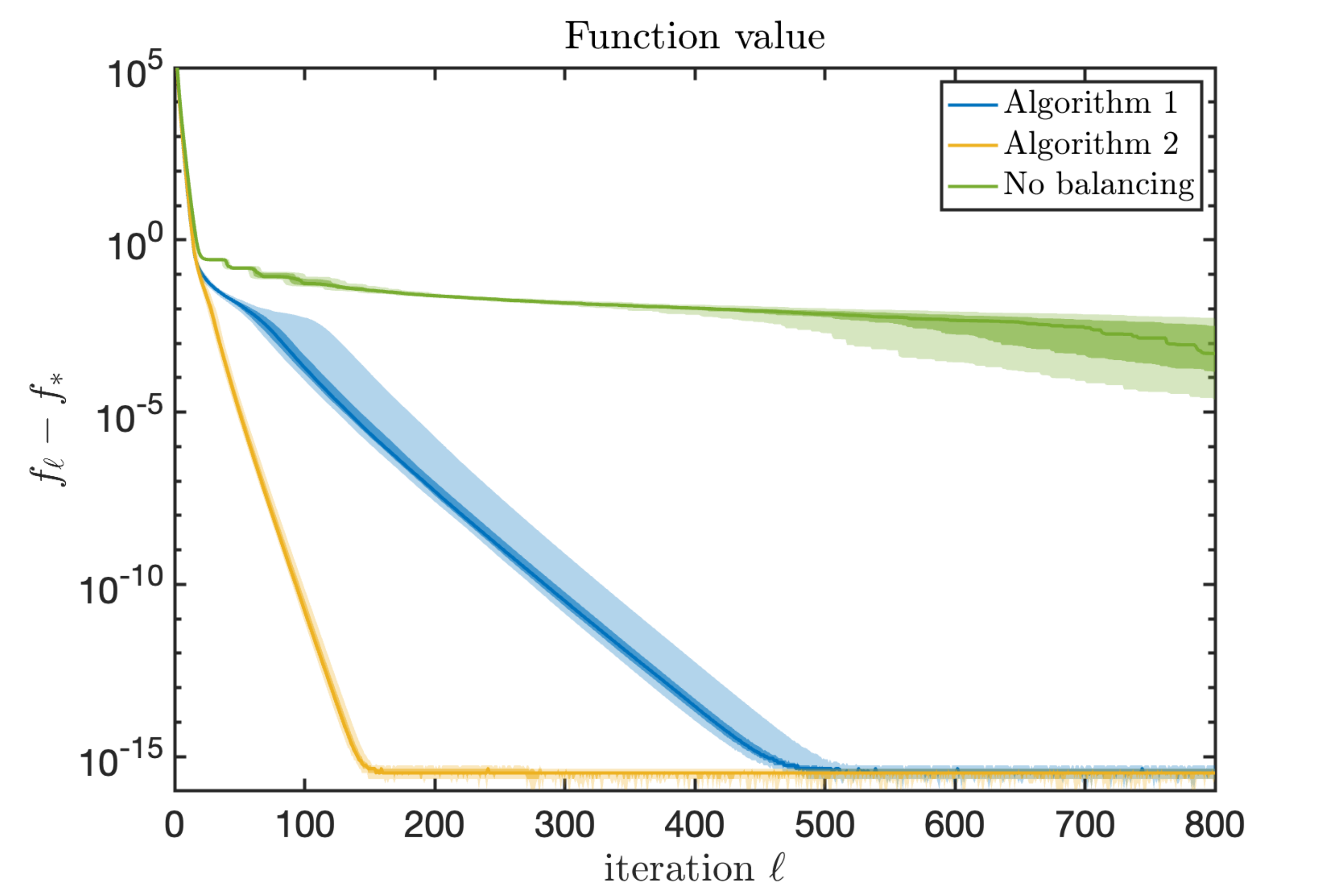}
    \end{minipage}%
    \begin{minipage}{0.5\textwidth}
        \centering
        \includegraphics[width=\linewidth]{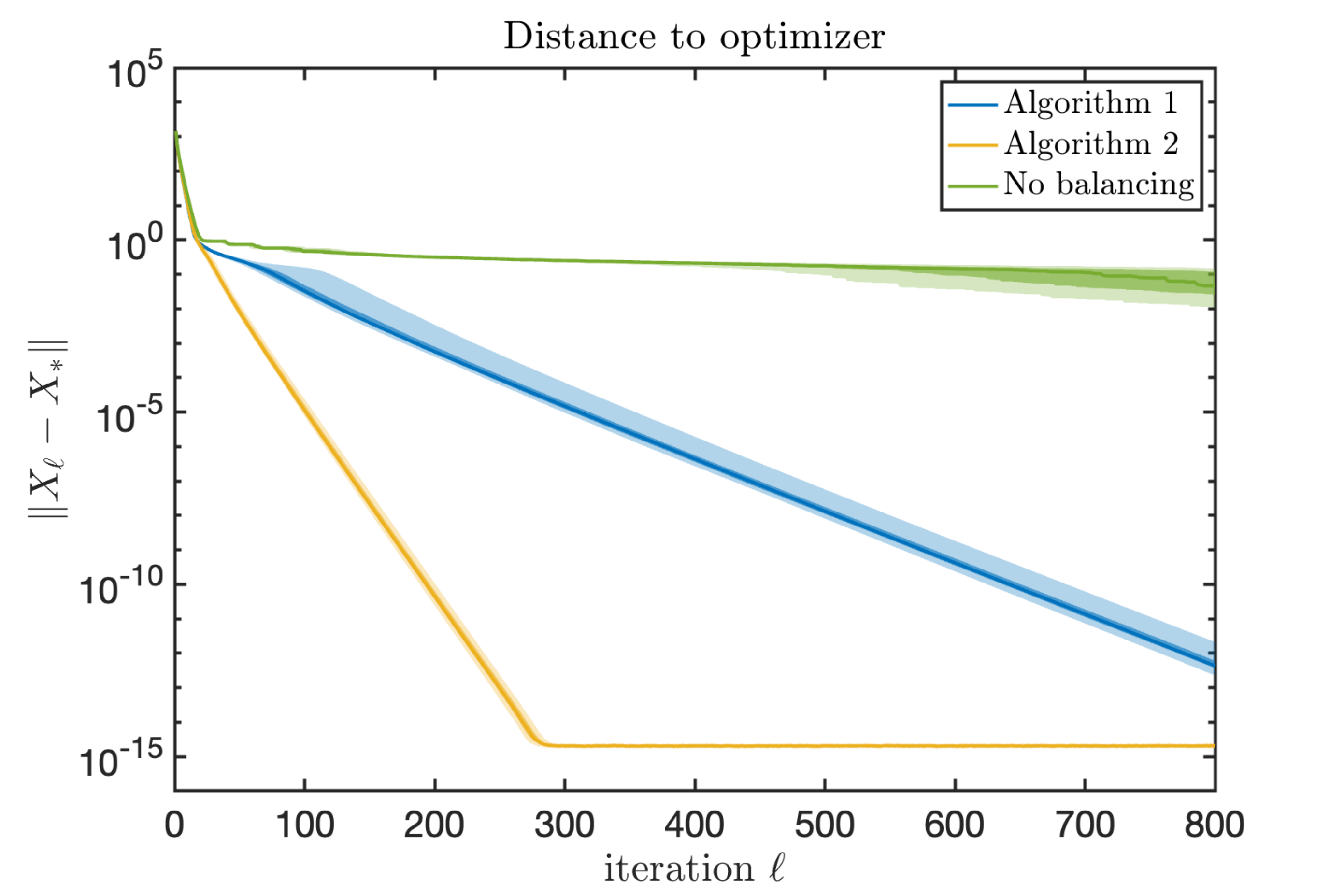}
    \end{minipage}
    \begin{minipage}{.5\textwidth}
        \centering
        \includegraphics[width=\linewidth]{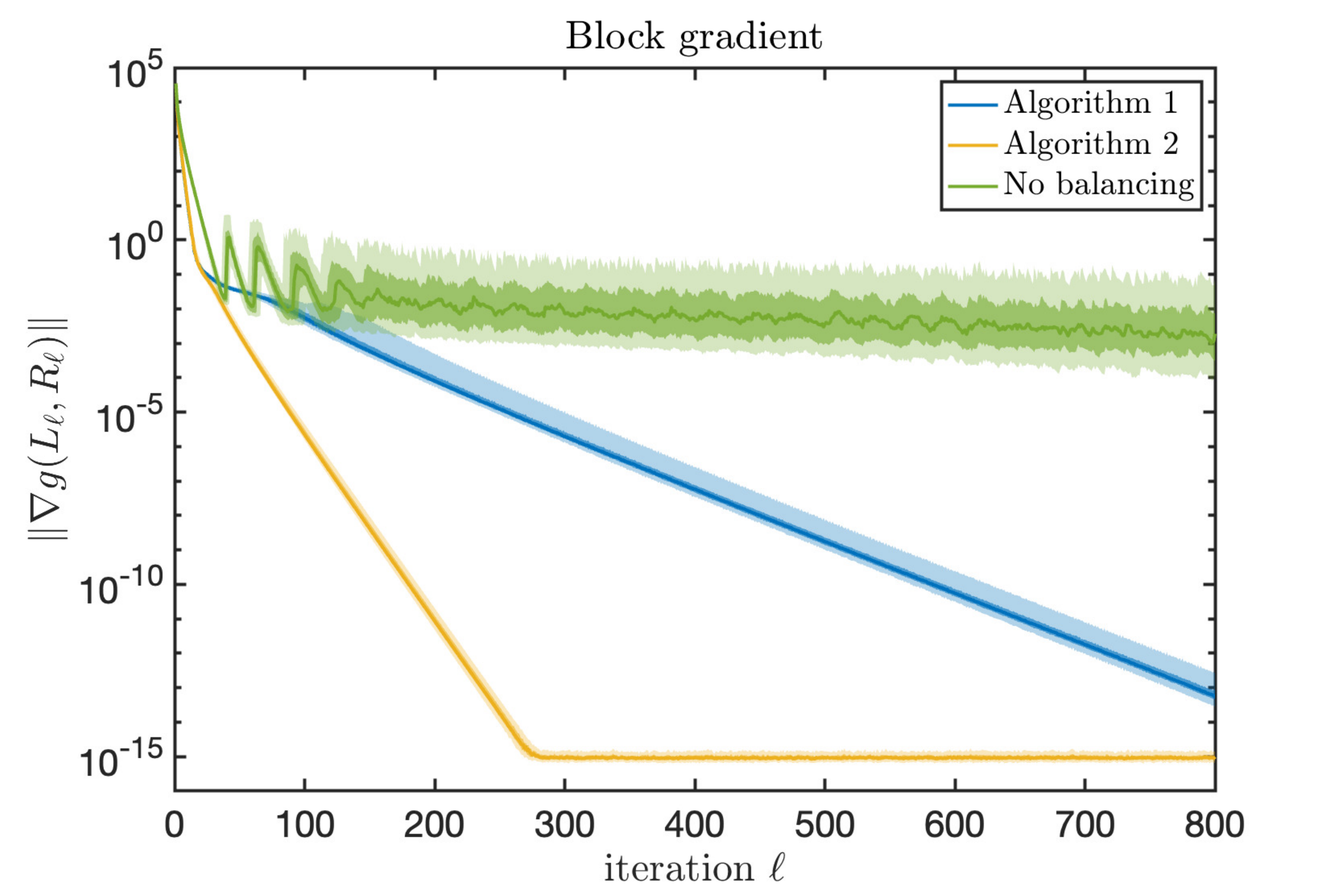}
    \end{minipage}%
    \begin{minipage}{0.5\textwidth}
        \centering
        \includegraphics[width=\linewidth]{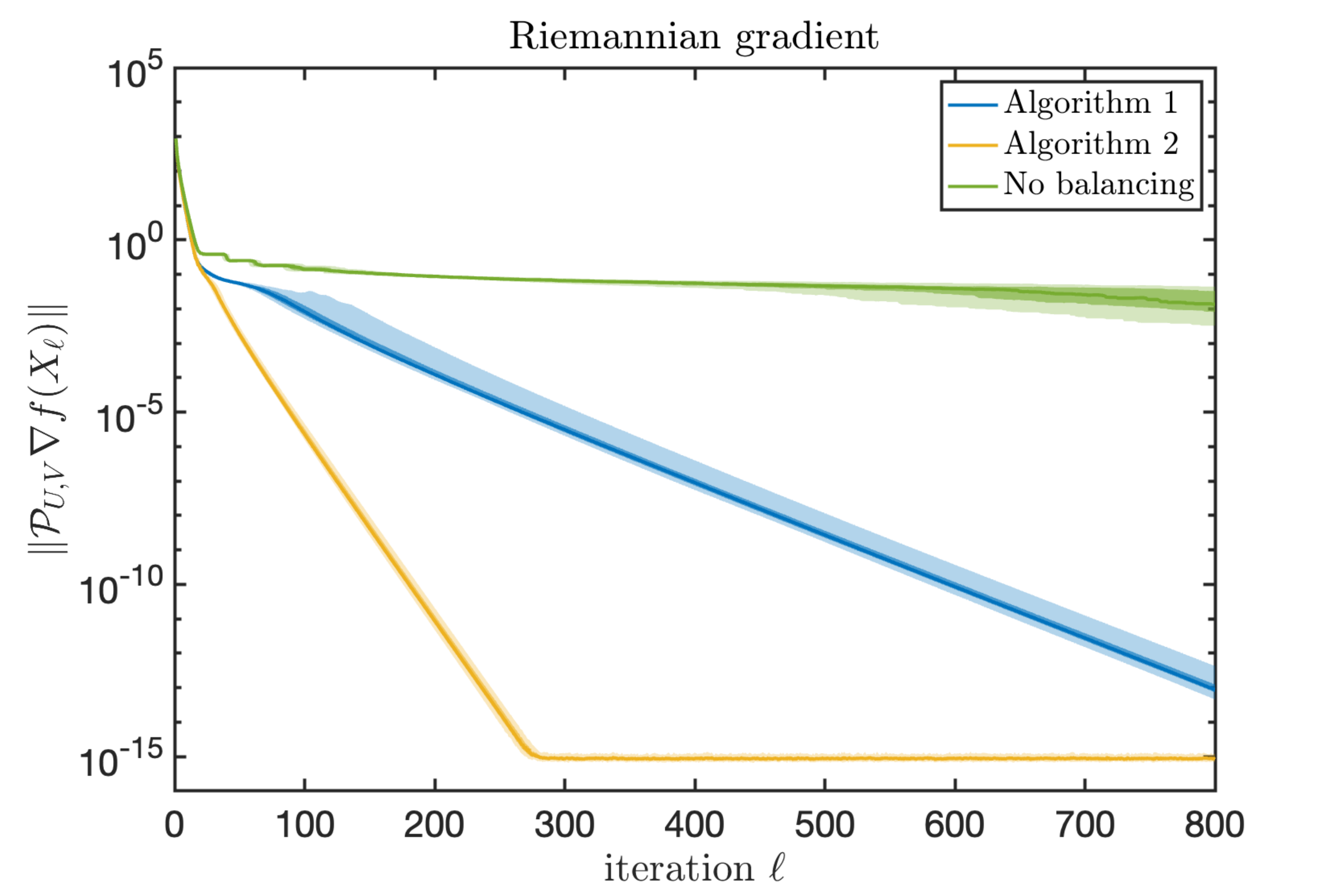}
    \end{minipage}
    \caption{Strongly convex quadratic function with condition number $\kappa_{\mathcal A} = 20$. The effective condition number of the minimizer $X_*$ is $\kappa_{X_*} = 5$. The shaded areas correspond to the 50 and 90 percentiles of the relevant quantity in each panel.}\label{fig:influence balancing}
\end{figure}

\vspace*{-1ex}
\paragraph{Influence of condition numbers $\kappa_{X_*}$ and $\kappa_{\mathcal A}$.}  As next experiment, we vary the condition numbers $\kappa_{X_*}$ and $\kappa_{\mathcal A}$. We no longer report on the version of Algorithm~\ref{algo: balanced GD} without balancing since it is clearly inferior. The size $n=30$ and rank $k=8$ are the same as above, and we take again 100 random realizations.

In Figure~\ref{fig:influence condX}, the condition number $\kappa_{\mathcal A}=20$ of $f$ is kept the same but the effective condition number $\kappa_{X_*}$ of the rank $k=8$ minimizer $X_*$ is varied from $1.1$ to $100$. We clearly see that the Riemannian Algorithm~\ref{algo: Subspace GD} is mostly unaffected by changes in $\kappa_{X_*}$, whereas the convergence for balancing (Algorithm~\ref{algo: balanced GD}) deteriorates with increasing $\kappa_{X_*}$. Quite remarkably, for $\kappa_{X_*}=1$ (not shown) both algorithms have essentially the same convergence behavior. For $\kappa_{X_*}=1.1$ already, Algorithm~\ref{algo: balanced GD} is slightly slower than Algorithm~\ref{algo: Subspace GD}. We can thus conclude that for this problem Algorithm~\ref{algo: balanced GD} is not robust with respect to the smallest singular values of $X_*$. Algorithm~\ref{algo: Subspace GD} however is robust.

\begin{figure}[t!]
    \centering
    \begin{minipage}{.5\textwidth}
        \centering
        \includegraphics[width=\linewidth]{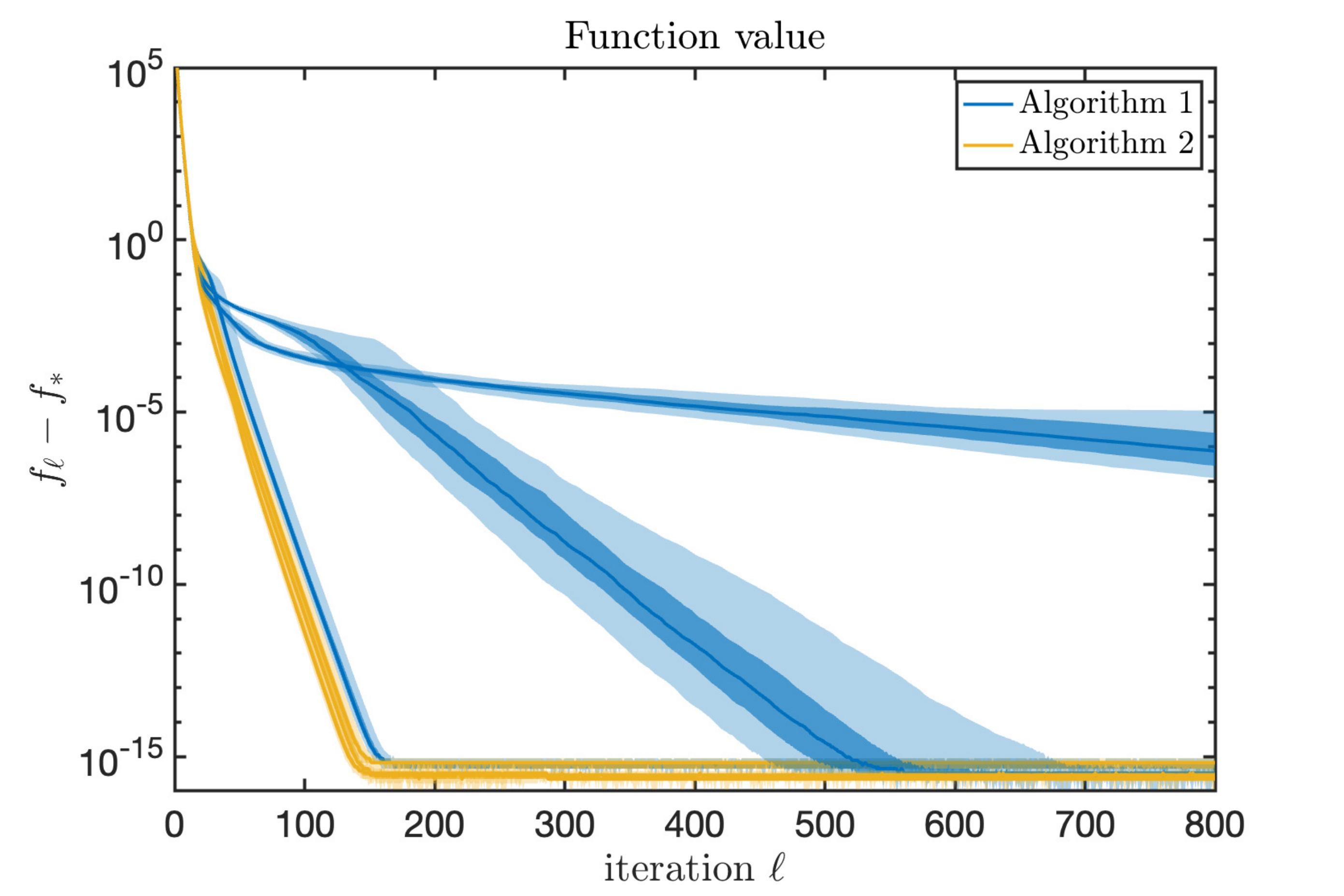}
    \end{minipage}%
    \begin{minipage}{0.5\textwidth}
        \centering
        \includegraphics[width=\linewidth]{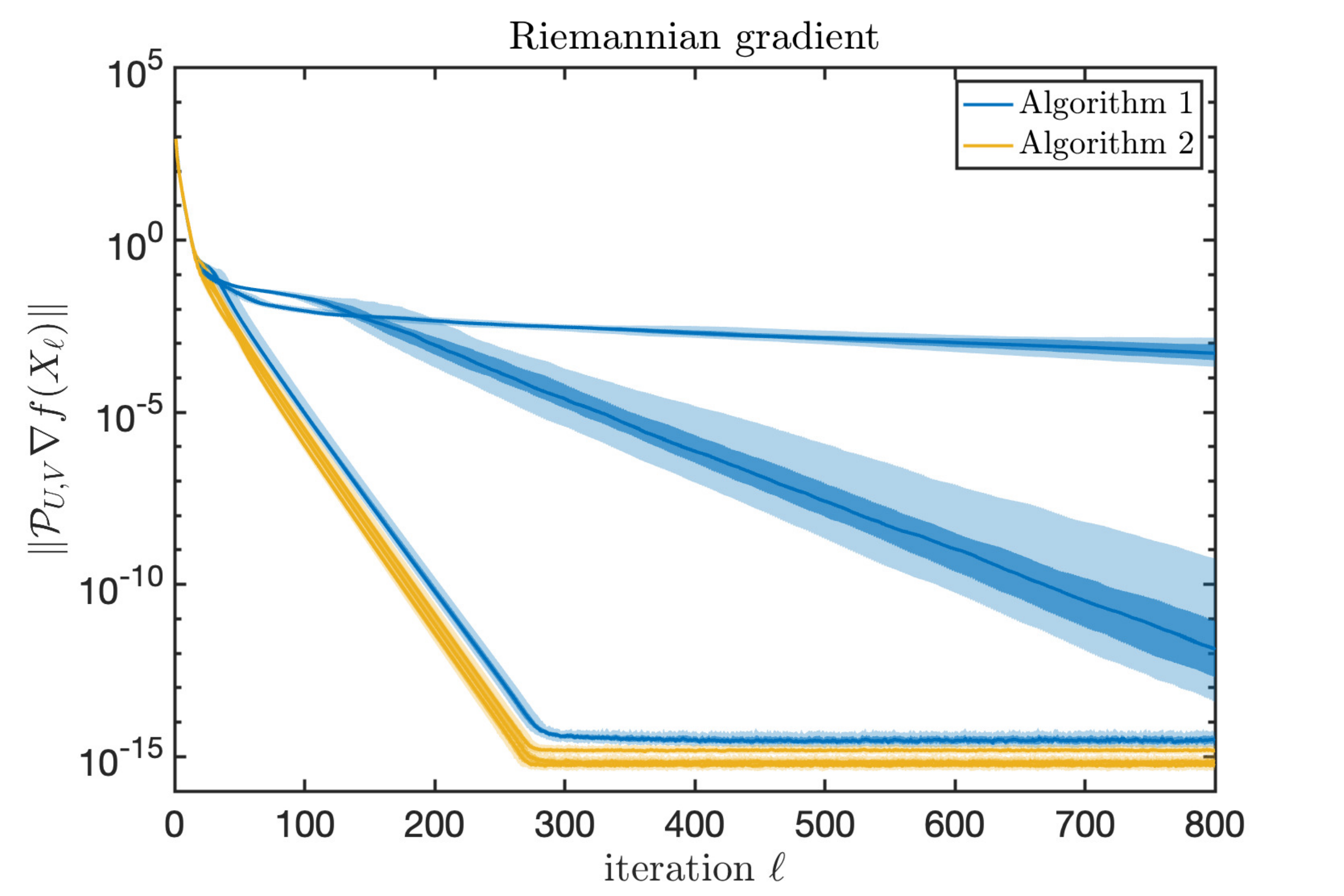}
    \end{minipage}
    \caption{Same setting as Fig.~\ref{fig:influence balancing} but now with several values of $\kappa_{X_*} \in \{ 1.1, 10, 100 \}$. The shaded areas correspond from left to right to increasing values of $\kappa_{X_*}$.}\label{fig:influence condX}
\end{figure}

As related experiment, we now keep $\kappa_{X_*}$ fixed but vary the condition number $\kappa_{\mathcal A}$ of $f$. In this case, it is normal that any algorithm that uses only first-order information will converge slower with higher $\kappa_{\mathcal A}$. Instead of a figure, Table~\ref{tab:iterations changing cond_A} lists the iteration numbers needed to achieve a reduction of the relative error in function value below a fixed tolerance of $\tau \coloneqq 10^{-10}$. In particular, for 100 random realizations of $f$ with the same $\kappa_{X_*}$, we compute the median of
\[ 
\ell_\tau \coloneqq \min \left\{ \ell \in \mathbb{N} \ \colon \ \frac{f(X_\ell) - f(X_*)}{f(X_0) - f(X_*)} < \tau \right\}.
\]

\begin{table}[t!]
\begin{center}
\begin{tabular}{ ccccc } 
 \hline
 & $\kappa_{\mathcal A}$ & 1   &  10  & 1000 \\ \toprule
\multirow{2}{*}{$\kappa_{X_*} = 1 $} & Alg.~\ref{algo: balanced GD}  & $12 \pm 1.5$  &  $25 \pm 4.5$ &  $98 \pm 30$\\
              & Alg.~\ref{algo: Subspace GD}  &   $3 \pm 0$ &   $18 \pm 1$  &  $85 \pm 6.5$ \\
 \midrule
 \multirow{2}{*}{$\kappa_{X_*} = 5 $} & Alg.~\ref{algo: balanced GD}  &  $20 \pm 4.5$  &  $45 \pm 10$ & $142 \pm 23$\\
              & Alg.~\ref{algo: Subspace GD}  &   $3 \pm 0$ &  $18 \pm 1$ &  $81\pm 6$ \\
 \bottomrule
\end{tabular}
\end{center}
\vspace*{-2ex}
\caption{Number of iterations $\ell_\tau$ needed to reduce the relative error in function value to at least $10^{-10}$. The first number shows the median over 100 random realizations and the second number is half the width of the 90 percentile interval.}
\label{tab:iterations changing cond_A}
\end{table}

\vspace*{-1ex}
\paragraph{Influence of line search.}  
Since the experiments from above all used exact line searches, we also tested fixed step sizes and Armijo backtracking. The experimental results are visible in Figure~\ref{fig:influence LS}.

\begin{figure}[t]
    \centering
    \begin{minipage}{.5\textwidth}
        \centering
        \includegraphics[width=\linewidth]{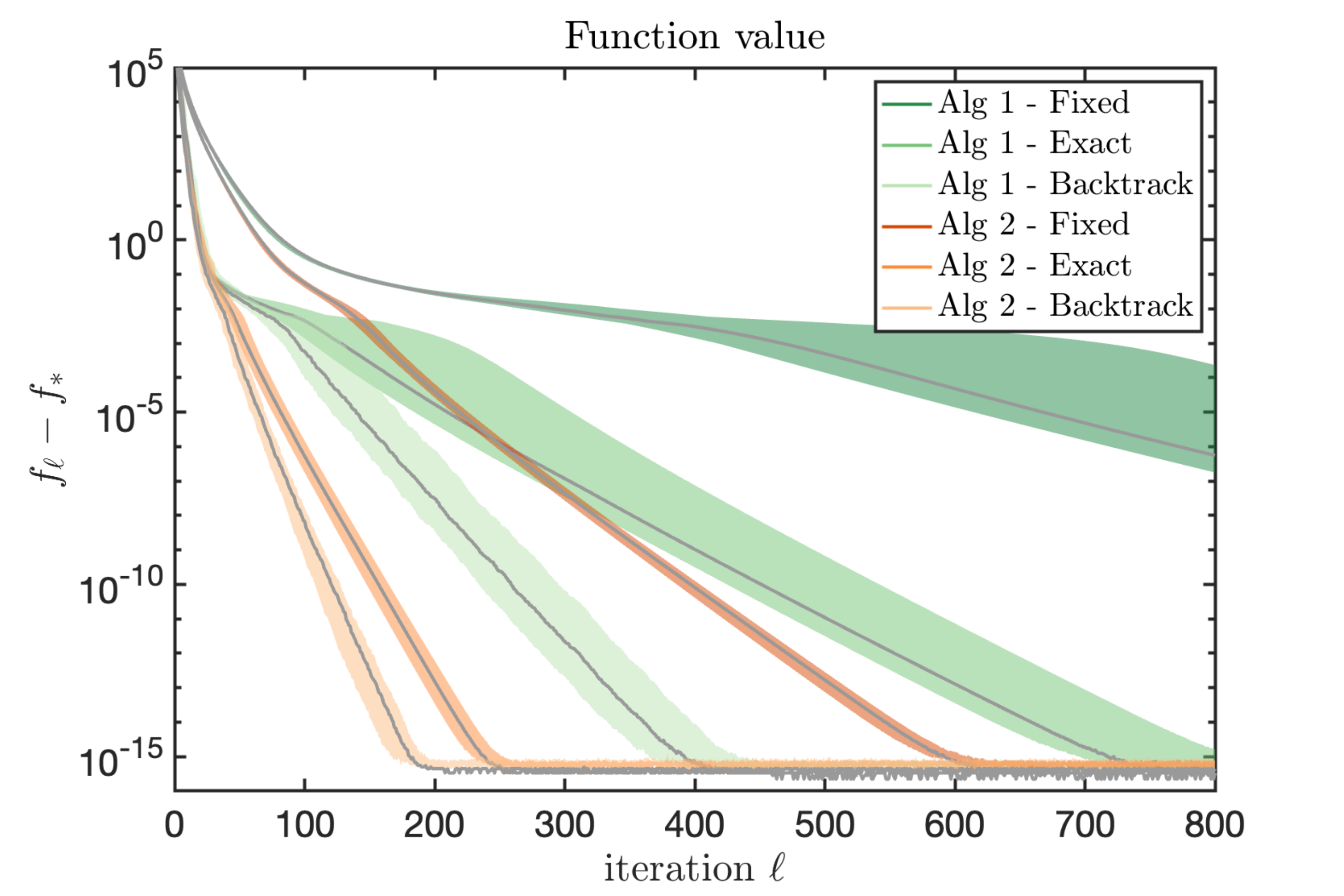}
    \end{minipage}%
    \begin{minipage}{0.5\textwidth}
        \centering
        \includegraphics[width=\linewidth]{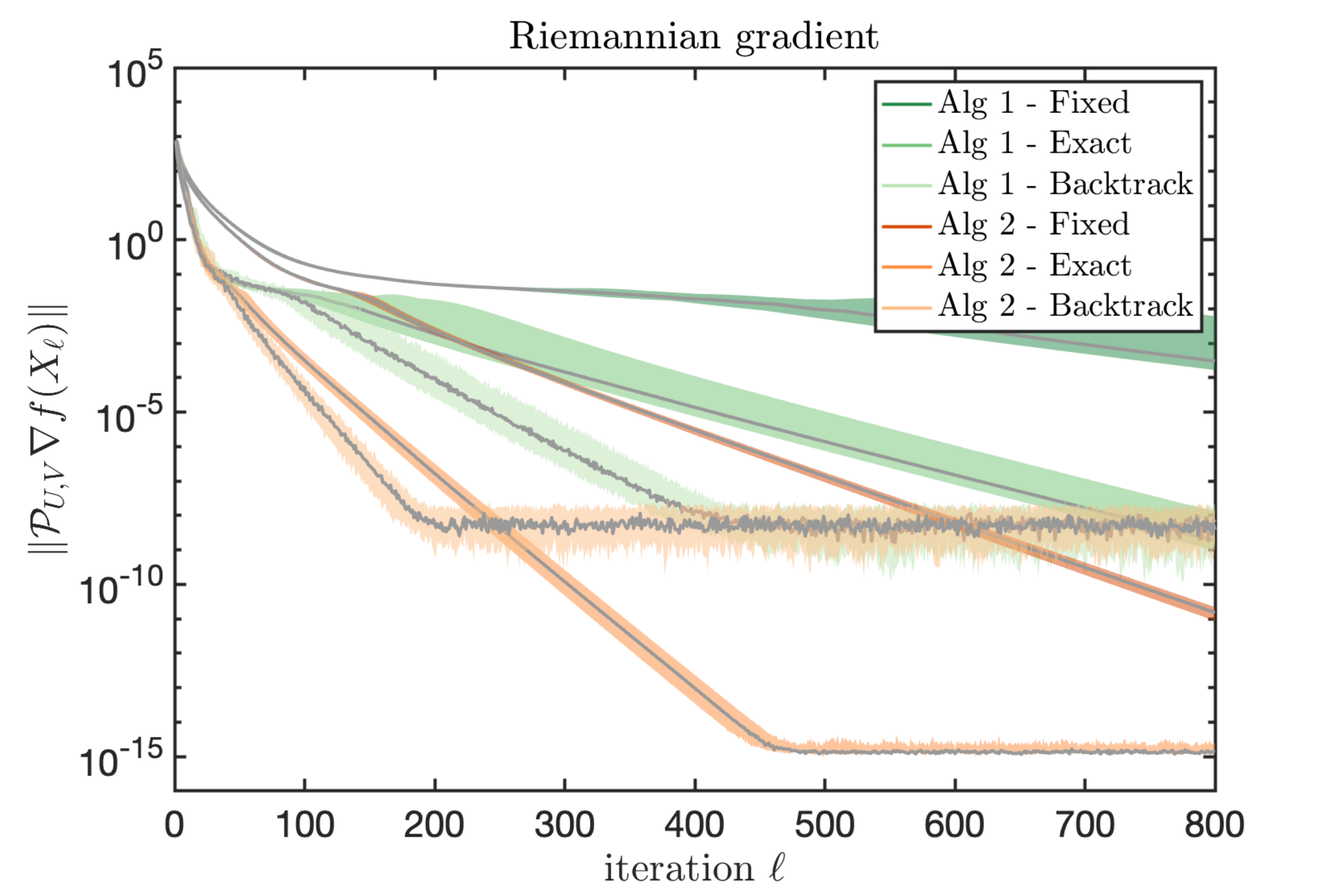}
    \end{minipage}
    \caption{Influence of the line search for a strongly convex quadratic function with fixed condition numbers $\kappa_{\mathcal A} = 40$ and $\kappa_{X_*} = 5$. The colored shaded areas correspond to the 90 percentile data over 20 random realizations. The grey lines indicate the median values.}\label{fig:influence LS}
\end{figure}

For the Armijo step-size rule, we use backtracking with the following parameters\footnote{The experimental results were not very sensitive to the choice of these parameters as long as $\bar{h}$ was sufficiently large.} as defined in Theorems~\ref{thm: balanced} and~\ref{thm: subspace}: initial step size $\bar{h} \coloneqq 10$, step reduction factor $\beta \coloneqq 0.5$, and sufficient-decrease factor $\gamma \coloneqq 10^{-4}$. For the exact step size, we use $\alpha_\ell \coloneqq 0.5$ in Theorems~\ref{thm: balanced} and~\ref{thm: subspace}. In addition, for the radius $\rho$ needed in Theorem~\ref{thm: balanced} we take $\|X_\ell\|_2$. While this is not practical, it gives the best performance for Algorithm~\ref{algo: balanced GD}. 

As expected, the fixed step-size rules lead to the slowest convergence since they are based on the worst-case global Lipschitz constant. Rather surprising, the Armijo step-size rule has the fastest convergence in terms of iterations. However, it needs on average about 10 function evaluations per iteration (but this can likely be improved with more advanced line searches based on interpolation and warm starting). On the other hand, the exact line search is only feasible for some problems like quadratic functions. Moreover, Armijo backtracking cannot decrease the Riemannian gradient much beyond the square root $\sqrt{\varepsilon_{\text{mach}}}$ of the machine precision. This is a well-known problem due to roundoff error when testing the sufficient-decrease condition in finite precision (but it could be mitigated to some extent by the techniques from~\cite{Sutti2021}). Finally, we see that for each choice of line search, Algorithm~\ref{algo: balanced GD} is slower than Algorithm~\ref{algo: Subspace GD}.

In summary, our experiments confirm several indications from the theoretical analysis that the Riemannian version of the Gauss--Southwell selection rule as realized by Algorithm~\ref{algo: Subspace GD} is superior to the factorized version based on balancing (Algorithm~\ref{algo: balanced GD}). This concerns the convergence rates as well as robustness against small singular values.

\paragraph*{Acknowledgements}
We thank P.-A.~Absil for valuable discussion on several parts of this work. We are also grateful to the anonymous referees for their valuable feedback. The work of G.O. was supported by the Fonds de la Recherche Scientifique -- FNRS and the Fonds Wetenschappelijk Onderzoek -- Vlaanderen under EOS Project no 30468160 and by the Fonds de la Recherche Scientifique -- FNRS under Grant no T.0001.23.
The work of A.U.~was supported by the Deutsche Forschungs\-gemeinschaft (DFG, German Research Foundation) – Projektnummer 506561557. The work of B.V.~was supported by the Swiss National Science Foundation (grant number 192129).

\small
\bibliographystyle{plain}
\bibliography{references.bib}

\end{document}